\documentclass[reqno]{amsart}
\usepackage{amsmath,amsthm,amssymb,amsfonts}
\usepackage{mathrsfs}
\numberwithin{equation}{section}
\usepackage{color}
\usepackage{enumerate}
\usepackage{multirow} 
\newtheoremstyle{mystyle}
{}
{}
{\normalfont}
{ }
{\bfseries}
{}
{10pt}
{ }
\theoremstyle{mystyle}
\newtheorem{theorem}{Theorem}

\newtheorem{lemma}{Lemma}
\newtheorem{remark}{Remark}

\usepackage{mathtools}
\usepackage[pdftex]{graphicx}

\def\Diag{\mathop{\rm Diag}\nolimits}
\def\tr{\mathop{\rm tr}\nolimits}
\def\vec{\mathop{\rm vec}\nolimits}
\def\vech{\mathop{\rm vech}\nolimits}
\def\det{\mathop{\rm det}\nolimits}
\def\rank{\mathop{\rm rank}\nolimits}

\newcommand{\dd}{\mathrm d}

\allowdisplaybreaks[4]
\usepackage[foot]{amsaddr}
\usepackage[top=3cm, bottom=3cm, left=2.5cm, right=2.5cm]{geometry}

\usepackage[hang,small,bf]{caption}
\usepackage[subrefformat=parens]{subcaption}
\usepackage{comment}
\large
\title[QBIC of SEM for diffusion processes]{QBIC of SEM for diffusion processes from discrete observations}
\author[S. Kusano]{Shogo Kusano $^{1}$}
\author[M. Uchida]{Masayuki Uchida $^{1,2}$}
\address{$^{1}$Graduate School of Engineering Science, Osaka University}
\address{$^{2}$Center for Mathematical Modeling and Data Science (MMDS), Osaka University and JST CREST}
\begin{document}
\begin{abstract}
\fontsize{8pt}{10pt}\selectfont
We deal with a model selection problem for structural equation modeling (SEM) with latent variables for diffusion processes. 
Based on the asymptotic expansion of the marginal quasi-log likelihood, we propose two types of quasi-Bayesian 
information criteria of the SEM. It is shown that the information criteria have model selection consistency. Furthermore, we examine the finite-sample performance of the proposed information criteria by numerical experiments.
\end{abstract}
\maketitle

\section{Introduction}
\fontsize{10pt}{16pt}\selectfont
We consider a model selection problem for structural equation modeling (SEM) with latent variables for diffusion processes. 
First, we set the true model of the SEM. The stochastic process $\mathbb{X}_{1,0,t}$ is defined as the following factor model:
\begin{align}
    \mathbb{X}_{1,0,t}&={\bf{\Lambda}}_{x_1,0}\xi_{0,t}+\delta_{0,t}\label{X0},
\end{align}
where $\{\mathbb{X}_{1,0,t}\}_{t\geq 0}$ is a $p_1$-dimensional observable vector process, $\{\xi_{0,t}\}_{t\geq 0}$ is a $k_{1}$-dimensional latent common factor vector process, $\{\delta_{0,t}\}_{t\geq 0}$ is a $p_{1}$-dimensional latent unique factor vector process, and ${\bf{\Lambda}}_{x_1,0}\in\mathbb{R}^{p_1\times k_1}$ is a constant loading matrix. $p_1$ is not zero, $k_{1}\leq p_1$, and $p_1$ and $k_{1}$ are fixed. Assume that $\{\xi_{0,t}\}_{t\geq 0}$ and $\{\delta_{0,t}\}_{t\geq 0}$ satisfy the following stochastic differential equations:
\begin{align}
    \quad\dd \xi_{0,t}&=B_{1}(\xi_{0,t})\dd t+{\bf{S}}_{1,0}\dd W_{1,t}, \ \ 
    \xi_{0,0}=c_{1}, \\
    \quad\dd\delta_{0,t}&=B_{2}(\delta_{0,t})\dd t+{\bf{S}}_{2,0}\dd W_{2,t}, \ \ 
    \delta_{0,0}=c_{2},     
\end{align}
where $B_{1}:\mathbb{R}^{k_{1}}\rightarrow\mathbb{R}^{k_{1}}$, ${\bf{S}}_{1,0}\in\mathbb{R}^{k_{1}\times r_{1}}$, $c_{1}\in\mathbb{R}^{k_{1}}$,
$B_{2}:\mathbb{R}^{p_1}\rightarrow\mathbb{R}^{p_1}$, ${\bf{S}}_{2,0}\in\mathbb{R}^{p_1\times r_{2}}$, $c_{2}\in\mathbb{R}^{p_1}$, and $W_{1,t}$ and $W_{2,t}$ are $r_{1}$ and $r_{2}$-dimensional standard Wiener processes, respectively. The stochastic process $\mathbb{X}_{2,0,t}$ is defined by the factor model as follows:
\begin{align}
    \mathbb{X}_{2,0,t}&={\bf{\Lambda}}_{x_2,0}\eta_{0,t}+\varepsilon_{0,t},
\end{align}
where $\{\mathbb{X}_{2,0,t}\}_{t\geq 0}$ is a $p_2$-dimensional observable vector process, $\{\eta_{0,t}\}_{t\geq 0}$ is a $k_{2}$-dimensional latent common factor vector process, $\{\varepsilon_{0,t}\}_{t\geq 0}$ is a $p_{2}$-dimensional latent unique factor vector process, and 
${\bf{\Lambda}}_{x_2,0}\in\mathbb{R}^{p_2\times k_{2}}$ is a constant loading matrix. $p_2$ is not zero, $k_{2}\leq p_2$, and $p_2$ and $k_{2}$ are fixed. Suppose that $\{\varepsilon_{0,t}\}_{t\geq 0}$ satisfies the following  stochastic differential equation:
\begin{align}
    \quad\dd\varepsilon_{0,t}&=B_{3}(\varepsilon_{0,t})\dd t+{\bf{S}}_{3,0}\dd W_{3,t}, \ \ 
     \varepsilon_{0,0}=c_{3},      
\end{align}
where $B_{3}:\mathbb{R}^{p_2}\rightarrow\mathbb{R}^{p_2}$, ${\bf{S}}_{3,0}\in\mathbb{R}^{p_2\times r_{3}}$, $c_{3}\in\mathbb{R}^{p_2}$,
and $W_{3,t}$ is an $r_{3}$-dimensional standard Wiener process. Set $p=p_1+p_2$. Moreover, the relationship between $\eta_{0,t}$ and  $\xi_{0,t}$ is expressed as follows:
\begin{align}
    \eta_{0,t}={\bf{B}}_0\eta_{0,t}+{\bf{\Gamma}}_0\xi_{0,t}+\zeta_{0,t},
\end{align}
where ${\bf{B}}_0\in\mathbb{R}^{k_{2}\times k_{2}}$ is a constant loading matrix, whose diagonal elements are zero, and ${\bf{\Gamma}}_0\in\mathbb{R}^{k_{2}\times k_{1}}$ is a constant loading matrix. $\{\zeta_{0,t}\}_{t\geq 0}$ is a $k_{2}$-dimensional latent unique factor vector process defined by the following stochastic differential equation: 
\begin{align}
    \quad\dd\zeta_{0,t}=B_{4}(\zeta_{0,t})\dd t+{\bf{S}}_{4,0}\dd W_{4,t},\ \ 
    \zeta_{0,0}=c_{4},\label{zeta0}
\end{align}
where $B_{4}:\mathbb{R}^{k_{2}}\rightarrow\mathbb{R}^{k_{2}}$, ${\bf{S}}_{4,0}\in\mathbb{R}^{k_{2}\times r_{4}}$, $c_{4}\in\mathbb{R}^{k_{2}}$, and $W_{4,t}$ is an $r_{4}$-dimensional standard Wiener process. $W_{1,t}$, $W_{2,t}$, $W_{3,t}$ and $W_{4,t}$ are independent standard Wiener processes on a stochastic basis with usual conditions $(\Omega, \mathscr{F}, \{\mathscr{F}_t\}, {\bf{P}})$. 
Let $\mathbb{X}_{0,t}=(\mathbb{X}_{1,0,t}^{\top},\mathbb{X}_{2,0,t}^{\top})^{\top}$, where $\top$ denotes the transpose.
Below, we simply write $\mathbb{X}_{0,t}$ as $\mathbb{X}_{t}$. $\mathbb{X}_n=(\mathbb{X}_{t_{i}^n})_{0\leq i\leq n}$ are discrete observations, where $t_{i}^n=ih_n$, $h_n=n^{-1}T$, $T>0$ is fixed, and $p_1$, $p_2$, $k_{1}$ and $k_{2}$ are independent of $n$. 
We consider 
a model selection problem among the following $M$ parametric models
based on discrete observations. Let $\theta_m\in\Theta_m$ be the parameter of Model $m\in\{1,\cdots,M\}$, where $\Theta_m$ is a convex and compact subset of $\mathbb{R}^{q_m}$ with a locally Lipschitz boundary; see, e.g., Adams and Fournier \cite{Adams(2003)}.
The stochastic process $\mathbb{X}^{\theta}_{1,m,t}$ is defined by the factor model as follows:
\begin{align}
    \mathbb{X}^{\theta}_{1,m,t}&={\bf{\Lambda}}^{\theta}_{x_1,m}\xi^{\theta}_{m,t}+\delta^{\theta}_{m,t}\label{Xm}, 
\end{align}
where $\{\mathbb{X}^{\theta}_{1,m,t}\}_{t\geq 0}$ is a $p_1$-dimensional observable vector process, $\{\xi^{\theta}_{m,t}\}_{t\geq 0}$ is a $k_{1}$-dimensional latent common factor vector process, $\{\delta^{\theta}_{m,t}\}_{t\geq 0}$ is a $p_{1}$-dimensional latent unique factor vector process, and ${\bf{\Lambda}}^{\theta}_{x_1,m}\in\mathbb{R}^{p_1\times k_{1}}$ is a constant loading matrix. Suppose that $\{\xi^{\theta}_{m,t}\}_{t\geq 0}$ and $\{\delta^{\theta}_{m,t}\}_{t\geq 0}$ satisfy the following  stochastic differential equations:
\begin{align}
    \quad\dd \xi^{\theta}_{m,t}&=B_{1}(\xi^{\theta}_{m,t})\dd t+{\bf{S}}^{\theta}_{1,m}\dd W_{1,t}, \ \ 
    \xi^{\theta}_{m,0}=c_{1},\label{xim}
    \\
    \quad\dd\delta^{\theta}_{m,t}&=B_{2}(\delta^{\theta}_{m,t})\dd t+{\bf{S}}^{\theta}_{2,m}\dd W_{2,t}, \ \ \delta^{\theta}_{m,0}=c_{2},\label{deltam}
\end{align}
where ${\bf{S}}^{\theta}_{1,m}\in\mathbb{R}^{k_{1}\times r_{1}}$ and ${\bf{S}}^{\theta}_{2,m}\in\mathbb{R}^{p_{1}\times r_{2}}$. The stochastic process $\mathbb{X}^{\theta}_{2,m,t}$ is defined as the following factor model:
\begin{align}
    \mathbb{X}^{\theta}_{2,m,t}={\bf{\Lambda}}^{\theta}_{x_2,m}\eta^{\theta}_{m,t}+\varepsilon^{\theta}_{m,t}\label{Ym}, 
\end{align}
where $\{\mathbb{X}^{\theta}_{2,m,t}\}_{t\geq 0}$ is a $p_2$-dimensional observable vector process, $\{\eta^{\theta}_{m,t}\}_{t\geq 0}$ is a $k_{2}$-dimensional latent common factor vector process, $\{\varepsilon^{\theta}_{m,t}\}_{t\geq 0}$ is a $p_2$-dimensional latent unique factor vector process, and ${\bf{\Lambda}}^{\theta}_{x_2,m}\in\mathbb{R}^{p_2\times k_{2}}$ is a constant loading matrix. Assume that $\{\varepsilon^{\theta}_{m,t}\}_{t\geq 0}$ satisfies the following stochastic differential equation:
\begin{align}
    \quad\dd\varepsilon^{\theta}_{m,t}&=B_{3}(\varepsilon^{\theta}_{m,t})\dd t+{\bf{S}}^{\theta}_{3,m}\dd W_{3,t}, \ \ \varepsilon^{\theta}_{m,0}=c_{3},\label{epsilonm}
\end{align}
where ${\bf{S}}^{\theta}_{3,m}\in\mathbb{R}^{p_{2}\times r_{3}}$. Furthermore, we express the relationship between $\eta^{\theta}_{m,t}$ and  $\xi^{\theta}_{m,t}$ as follows:
\begin{align}
    \eta^{\theta}_{m,t}={\bf{B}}_m^{\theta}\eta^{\theta}_{m,t}+{\bf{\Gamma}}_m^{\theta}\xi^{\theta}_{m,t}+\zeta^{\theta}_{m,t}\label{etam},
\end{align}
where ${\bf{B}}^{\theta}_m\in\mathbb{R}^{k_{2}\times k_{2}}$ is a constant loading matrix, whose diagonal elements are zero, and ${\bf{\Gamma}}^{\theta}_m\in\mathbb{R}^{k_{2}\times k_{1}}$ is a constant loading matrix. $\{\zeta^{\theta}_{m,t}\}_{t\geq 0}$ is a $k_2$-dimensional latent unique factor vector process defined as the following stochastic differential equation:
\begin{align}
    \quad\dd\zeta^{\theta}_{m,t}=B_{4}(\zeta^{\theta}_{m,t})\dd t+{\bf{S}}^{\theta}_{4,m}\dd W_{4,t}, \ \ 
    \zeta^{\theta}_{m,0}=c_{4},\label{zetam}
\end{align}
where ${\bf{S}}^{\theta}_{4,m}\in\mathbb{R}^{k_{2}\times r_{4}}$. Let ${\bf{\Psi}}_0=\mathbb{I}_{k_2}-{\bf{B}}_0$ and ${\bf{\Psi}}^{\theta}_m=\mathbb{I}_{k_2}-{\bf{B}}^{\theta}_m$, where $\mathbb{I}_{k_2}$ denotes the identity matrix of size $k_2$. Set ${\bf{\Sigma}}_{\xi\xi,0}={\bf{S}}_{1,0}{\bf{S}}_{1,0}^{\top}$, ${\bf{\Sigma}}_{\delta\delta,0}={\bf{S}}_{2,0}{\bf{S}}_{2,0}^{\top}$, ${\bf{\Sigma}}_{\varepsilon\varepsilon,0}={\bf{S}}_{3,0}{\bf{S}}_{3,0}^{\top}$, ${\bf{\Sigma}}_{\zeta\zeta,0}={\bf{S}}_{4,0}{\bf{S}}_{4,0}^{\top}$, ${\bf{\Sigma}}^{\theta}_{\xi\xi,m}={\bf{S}}^{\theta}_{1,m}{\bf{S}}^{\theta\top}_{1,m}$, ${\bf{\Sigma}}^{\theta}_{\delta\delta,m}={\bf{S}}^{\theta}_{2,m}{\bf{S}}^{\theta\top}_{2,m}$, ${\bf{\Sigma}}^{\theta}_{\varepsilon\varepsilon,m}={\bf{S}}^{\theta}_{3,m}{\bf{S}}^{\theta\top}_{3,m}$ and ${\bf{\Sigma}}^{\theta}_{\zeta\zeta,m}={\bf{S}}^{\theta}_{4,m}{\bf{S}}^{\theta\top}_{4,m}$.
It is assumed that ${\bf{\Lambda}}_{x_1,0}$ and ${\bf{\Lambda}}^{\theta}_{x_1,m}$ are full column rank matrices, ${\bf{\Psi}}_0$ and ${\bf{\Psi}}_m^{\theta}$ are non-singular, and ${\bf{\Sigma}}_{\delta\delta,0}$, ${\bf{\Sigma}}_{\varepsilon\varepsilon,0}$, ${\bf{\Sigma}}^{\theta}_{\delta\delta}$ and ${\bf{\Sigma}}^{\theta}_{\varepsilon\varepsilon}$ are positive definite matrices. Let $\mathbb{X}^{\theta}_{m,t}=(\mathbb{X}_{1,m,t}^{\theta\top},\mathbb{X}_{2,m,t}^{\theta\top})^{\top}$. The variances of $\mathbb{X}_{t}$ and $\mathbb{X}^{\theta}_{m,t}$ are defined as
\begin{align*}
    \qquad {\bf{\Sigma}}_0=\begin{pmatrix}
    {\bf{\Sigma}}_0^{11} & {\bf{\Sigma}}_0^{12}\\
    {\bf{\Sigma}}_0^{12\top} & {\bf{\Sigma}}_0^{22}
    \end{pmatrix},\qquad 
    {\bf{\Sigma}}_m(\theta_m)=\begin{pmatrix}
    {\bf{\Sigma}}_m^{11}(\theta_m) & {\bf{\Sigma}}_m^{12}(\theta_m)\\
    {\bf{\Sigma}}_m^{12}(\theta_m)^{\top} & {\bf{\Sigma}}_m^{22}(\theta_m)
    \end{pmatrix}
\end{align*}
respectively, where 
\begin{align*} 
    \qquad\qquad{\bf{\Sigma}}^{11}_0&
    ={\bf{\Lambda}}_{x_1,0}{\bf{\Sigma}}_{\xi\xi,0}{\bf{\Lambda}}_{x_1,0}^{\top}
    +{\bf{\Sigma}}_{\delta\delta,0},\\  {\bf{\Sigma}}^{12}_0&={\bf{\Lambda}}_{x_1,0}{\bf{\Sigma}}_{\xi\xi,0}{\bf{\Gamma}}_0^{\top}{\bf{\Psi}}_0^{-1\top}{\bf{\Lambda}}_{x_2,0}^{\top},\\
    {\bf{\Sigma}}^{22}_0&={\bf{\Lambda}}_{x_2,0}{\bf{\Psi}}_0^{-1}({\bf{\Gamma}}_0{\bf{\Sigma}}_{\xi\xi,0}{\bf{\Gamma}}_0^{\top}+{\bf{\Sigma}}_{\zeta\zeta,0}){\bf{\Psi}}_0^{-1\top}{\bf{\Lambda}}_{x_2,0}^{\top}+{\bf{\Sigma}}_{\varepsilon\varepsilon,0}
\end{align*}
and
\begin{align*}
    \qquad\qquad{\bf{\Sigma}}^{11}_m(\theta_m)&
    ={\bf{\Lambda}}^{\theta}_{x_1,m}{\bf{\Sigma}}^{\theta}_{\xi\xi,m}{\bf{\Lambda}}_{x_1,m}^{\theta\top}
    +{\bf{\Sigma}}^{\theta}_{\delta\delta,m},\\  {\bf{\Sigma}}^{12}_m(\theta_m)&={\bf{\Lambda}}^{\theta}_{x_1,m}{\bf{\Sigma}}^{\theta}_{\xi\xi,m}{\bf{\Gamma}}_m^{\theta\top}{\bf{\Psi}}_m^{\theta-1\top}{\bf{\Lambda}}_{x_2,m}^{\theta\top},\\
    {\bf{\Sigma}}^{22}_m(\theta_m)&={\bf{\Lambda}}^{\theta}_{x_2,m}{\bf{\Psi}}_m^{\theta-1}({\bf{\Gamma}}_m^{\theta}{\bf{\Sigma}}^{\theta}_{\xi\xi,m}{\bf{\Gamma}}_m^{\theta\top}+{\bf{\Sigma}}^{\theta}_{\zeta\zeta,m}){\bf{\Psi}}_m^{\theta-1\top}{\bf{\Lambda}}_{x_2,m}^{\theta\top}+{\bf{\Sigma}}^{\theta}_{\varepsilon\varepsilon,m}.
\end{align*}
Furthermore, we suppose that Model $m$ is a correctly specified model. In other words, it is assumed that there exists 
$\theta_{m,0}\in{\rm{Int}}(\Theta_m)$ such that ${\bf{\Sigma}}_0={\bf{\Sigma}}_m(\theta_{m,0})$.

SEM with latent variables is a method for examining the relationships among latent variables that cannot be observed. 
In SEM, statisticians often consider some candidate models.
In this case, it is necessary to select the optimal model among these models. 
In SEM, the Akaike information criterion (AIC) and Bayesian information criterion (BIC) are often used for model selection; see e.g., Huang \cite{Huang AIC(2017)}. These criteria are devised for different purposes. 
AIC selects the model to minimize Kullback-Leibler divergence in terms of prediction. On the other hand, BIC selects the model to maximize posterior probability given the data. As is well known, BIC has the model selection consistency while AIC does not. 

The statistical inference for continuous-time stochastic processes based on discrete observations has been 
investigated for many years,
see  Yoshida \cite{Yoshida(1992)}, Genon-Catalot and Jacod \cite{Genon(1993)}, Bibby and S{\o}rensen \cite{Bibby(1995)},  Kessler \cite{kessler(1997)},  Yoshida \cite{Yoshida(2011)}, Uchida and Yoshida \cite{Uchi-Yoshi(2012)} \cite{Uchi-Yoshi(2014)}, Kamatani and Uchida \cite{Kamatani(2015)} and references therein.
In the field of financial econometrics, many researchers have extensively considered the factor model based on high-frequency data. For example, A{\"i}t-Sahalia and Xiu \cite{Ait(2017)} studied a continuous-time latent factor model for a high-dimensional model using principal component analysis, and Kusano and Uchida \cite{Kusano(2024a)} extended the classical factor analysis to diffusion processes. Using these methods, we can investigate the relationships between observed variables and latent variables based on high-frequency data. 
However, the factor model does not provide the relationships between latent variables. Thus, Kusano and Uchida \cite{Kusano(JJSD)} proposed SEM with latent variables for diffusion processes, and studied the asymptotic properties. The method enables us to study the relationships between latent variables based on high-frequency data. See also Kusano and Uchida \cite{Kusano(2024c)} for details of the SEM. 

Recently, the model selection problem for continuous-time stochastic processes based on high-frequency data have studied; see, e.g., Uchida \cite{Uchida(2010)}, Fuji and Uchida \cite{Fuji(2014)}, Eguchi and Uehara \cite{Eguchi(2021)}, Eguchi and Masuda \cite{Eguchi(2018)} and references therein. Eguchi and Masuda \cite{Eguchi(2018)} proposed the quasi-BIC (QBIC) by using the stochastic expansion of the maginal quasi-likelihood, and applied it to ergodic diffusion processes and continuous semimartingales. Eguchi and Masuda \cite{Eguchi(2018)} also showed that the QBIC has the model selection consistency. Furthermore, Eguchi and Masuda \cite{Eguchi(2023)} studied the Gaussian QBIC for the ergodic  L$\acute{e}$vy driven processes. 

Kusano and Uchida \cite{Kusano(2024c)} studied the quasi-AIC (QAIC) for the SEM with latent variables for diffusion processes, and examined the asymptotic properties. On the other hand, what we propose in this paper is the QBIC for the SEM . It is shown that the QBIC for the SEM has the model selection consistency unlike the QAIC for the SEM.

This paper is organized as follows. In Section $2$, the notation and assumptions are presented. In Section $3$, based on the asymptotic expansion of the marginal quasi-log likelihood, 
we propose two types of quasi-Bayesian information criteria 
of the SEM, and investigate their asymptotic properties. In Section $4$, we provide examples and simulation studies 
to show the finite-sample performance of the proposed information criteria. Section $5$ is devoted to the proofs of the theorems given in Section $3$. 
\section{Notation and assumptions}
First, we prepare the notations and definitions as follows. For any vector $v$, $|v|=\sqrt{\tr{vv^\top}}$, $v^{(i)}$ is the $i$-th element of $v$, and $\Diag v$ is the diagonal matrix, whose $i$-th diagonal element is $v^{(i)}$.
For any matrix $A$, $|A|=\sqrt{\tr{AA^\top}}$, $A^{\otimes 2}=AA^{\top}$, $A_{ij}$ is the $(i,j)$-th element of $A$, and $A^{+}$ represents the Moore-Penrose inverse of $A$. If $A$ is a positive definite matrix, 
we write $A>0$.
For any symmetric matrix $A\in\mathbb{R}^{p\times p}$, $\vec A$, $\vech A$ and $\mathbb{D}_{p}$ stand for the vectorization of $A$, the half-vectorization of $A$ and the $p^2\times\bar{p}$ duplication matrix respectively, where $\bar{p}=p(p+1)/2$; see, e.g., Harville \cite{Harville(1998)}. For matrices $A$ and $B$ of the same size, we define $A[B]=\tr(AB^{\top})$. Let $\partial_{\theta_m}=\partial\slash\partial\theta_m$ and $\partial^2_{\theta_m}=\partial_{\theta_m}\partial_{\theta_m}^{\top}$. Set $ C^{k}_{\uparrow}(\mathbb R^{d})$ as the space of all functions $f$ satisfying the following conditions:
\begin{itemize}
    \item[(i)] $f$ is continuously differentiable with respect to $x\in \mathbb{R}^d$ up to order $k$. 
    \item[(ii)] $f$ and all its derivatives are of polynomial growth in $x\in \mathbb{R}^d$, i.e., 
    $g$ is
of polynomial growth in $x\in \mathbb{R}^d$ if $\displaystyle g(x)=R(1,x)$. 
\end{itemize}
$\mathcal{M}^{+}_p$ is the set of all $p\times p$ real-valued positive definite matrices. 
${\bf{E}}$ denotes the expectation under ${\bf{P}}$. $\stackrel{p}{\longrightarrow}$ and $\stackrel{d}{\longrightarrow}$ denote the convergence in probability and the convergence in distribution, respectively. Define
\begin{align*}
    \mathbb{Q}_{\mathbb{XX}}=\frac{1}{T}\sum_{i=1}^n(\mathbb{X}_{t_i^n}-\mathbb{X}_{t_{i-1}^n})(\mathbb{X}_{t_i^n}-\mathbb{X}_{t_{i-1}^n})^{\top}.
\end{align*}
Set the prior density of $\theta_m$ as $\pi_{m,n}(\theta_m)$. Let 
\begin{align*}
    \mathbb{U}_{m,n}(\theta_{m,0})=\Bigl\{u_m\in\mathbb{R}^{q_m}:\theta_{m,0}+n^{-\frac{1}{2}}u_m\in\Theta_m\Bigr\}.
\end{align*}

Next, we make the following assumptions.
\begin{enumerate}
    \vspace{2mm}
    \item[\bf{[A]}]
    \begin{enumerate}
    \item
    \begin{enumerate}
    \item[(i)] There exists a constant $C>0$ such that
    \begin{align*}
    |B_{1}(x)-B_{1}(y)|\leq C|x-y|
    \end{align*}
    for any $x,y\in\mathbb R^{k_{1}}$.
    \vspace{2mm}
    \item[(ii)] For all $\ell>0$,
    $\displaystyle\sup_t{\bf{E}}\Bigl[\bigl|\xi_{0,t}\bigr|^{\ell}\Bigr]<\infty$.
    \vspace{2mm}
    \item[(iii)] $B_{1}\in C^{4}_{\uparrow}(\mathbb R^{k_{1}})$.
    \end{enumerate} 
    \vspace{2mm}
    \item
    \begin{enumerate}
    \item[(i)] There exists a constant $C>0$ such that
    \begin{align*}
    |B_{2}(x)-B_{2}(y)|\leq C|x-y|
    \end{align*}
    for any $x,y\in\mathbb R^{p_1}$.
    \vspace{2mm}
    \item[(ii)] For all $\ell\geq 0$, 
    $\displaystyle\sup_t{\bf{E}}\Bigl[\bigl|\delta_{0,t}\bigr|^{\ell}\Bigr]<\infty$.
    \vspace{2mm}
    \item[(iii)] $B_{2}\in C^{4}_{\uparrow}(\mathbb R^{p_1})$.
    \end{enumerate}
    \vspace{2mm}
    \item
    \begin{enumerate}
    \item[(i)] There exists a constant $C>0$ such that 
    \begin{align*}
    |B_{3}(x)-B_{3}(y)|\leq C|x-y|
    \end{align*}
    for any $x,y\in\mathbb R^{p_2}$.
    \vspace{2mm}
    \item[(ii)] For all $\ell\geq 0$, 
    $\displaystyle\sup_t{\bf{E}}\Bigl[\bigl|\varepsilon_{0,t}\bigr|^{\ell}\Bigr]<\infty$.
    \vspace{2mm}
    \item[(iii)] $B_{3}\in C^{4}_{\uparrow}(\mathbb R^{p_2})$.
    \end{enumerate}
    \vspace{2mm}
    \item
    \begin{enumerate}
    \item[(i)] There exists a constant $C>0$ such that
    \begin{align*}
    |B_{4}(x)-B_{4}(y)|\leq C|x-y|
    \end{align*}
    for any $x,y\in\mathbb R^{k_{2}}$.
    \vspace{2mm}
    \item[(ii)] For all $\ell\geq 0$, $\displaystyle\sup_t{\bf{E}}\Bigl[\bigl|\zeta_{0,t}\bigr|^{\ell}\Bigr]<\infty$.
    \vspace{2mm}
    \item[(iii)] $B_{4}\in C^{4}_{\uparrow}(\mathbb R^{k_{2}})$.
    \end{enumerate}
    \end{enumerate}
    \vspace{2mm}
    \item[\bf{[B]}]
    \begin{enumerate}
        \item For $m=1,\cdots,M$, 
        \begin{align*}
            \inf_{n\in\mathbb{N}}\pi_{m,n}(\theta_{m,0})>0,\quad 
            \sup_{\theta_m\in\Theta_m}
            \sup_{n\in\mathbb{N}}\pi_{m,n}(\theta_m)<\infty.
        \end{align*}
        \item For $m=1,\cdots,M$,
        \begin{align*}
            \sup_{u_m\in\{|u_m|<L\}\cap\mathbb{U}_{m,n}(\theta_{m,0})}\Bigl|\pi_{m,n}\bigl(\theta_{m,0}+n^{-\frac{1}{2}}u_m\bigr)-\pi_{m,n}(\theta_{m,0})\Bigr|\longrightarrow 0
        \end{align*}
    as $n\longrightarrow\infty$ for each $L>0$.
    \end{enumerate}
\end{enumerate}
\begin{remark}
{\bf{[A]}} is a standard assumption for diffusion processes; see, e.g., Kessler \cite{kessler(1997)}. {\bf{[B]}} is the condition of the prior distribution $\pi_{m,n}$. For details of {\bf{[B]}}, see Eguchi and Masuda \cite{Eguchi(2018)}. 
\end{remark}
\section{Main results}
Based on the discrete observations $\mathbb{X}_n$, we define the quasi-likelihood $\mathbb{L}_{m,n}$ as
{\setlength{\abovedisplayskip}{8pt}
\setlength{\belowdisplayskip}{8pt}
\begin{align*}
    \mathbb{L}_{m,n}(\theta_m)=\prod_{i=1}^n\frac{1}{(2\pi h_n)^{\frac{p}{2}}\det{{\bf{\Sigma}}_m
    (\theta_m)^{\frac{1}{2}}}}\exp\biggl(-\frac{1}{2h_n}{\bf{\Sigma}}_m(\theta_m)
    ^{-1}\Bigl[\bigl(\Delta_i\mathbb{X}\bigr)^{\otimes 2}\Bigr]\biggr),
\end{align*}}
where $\Delta_i\mathbb{X}=\mathbb{X}_{t_i^n}-\mathbb{X}_{t_{i-1}^n}$.
Set the quasi-maximum likelihood estimator $\hat{\theta}_{m,n}$ as
\begin{align*}
    \mathbb{L}_{m,n}(\hat{\theta}_{m,n})=\sup_{\theta_m\in\Theta_m}
    \mathbb{L}_{m,n}(\theta_m).
\end{align*}
Furthermore, the marginal quasi-likelihood is defined as
\begin{align*}
    h_{m,n}=\int_{\Theta_m}\exp{\bigl\{\mathbb{H}_{m,n}(\theta_m)\bigr\}}\pi_{m,n}(\theta_m)d\theta_m,
\end{align*}
where
\begin{align*}
    \mathbb{H}_{m,n}(\theta_m)&=\log (2\pi h_n)^{\frac{np}{2}}\mathbb{L}_{m,n}(\theta_m).
\end{align*}
Set the relative likeliness of Model $m$ occurrence as $\mathfrak{p}_m>0$, where $\sum_{i=1}^M\mathfrak{p}_i=1$.
The posterior model selection probability of Model $m$ is given as
\begin{align}
    \frac{\mathfrak{p}_mh_{m,n}}{\sum_{i=1}^M\mathfrak{p}_ih_{i,n}}
    \label{post}
\end{align}
for $m=1,\cdots,M$. Our purpose is to know the model which maximizes the posterior model selection probability (\ref{post}). As $\mathfrak{p}_1=\mathfrak{p}_2=\cdots=\mathfrak{p}_M$, the model which maximizes (\ref{post}) is equivalent to the model which maximizes $h_{m,n}$. Consequently, it is sufficient to estimate $h_{m,n}$. However, $h_{m,n}$ is very intractable. Thus, we consider the asymptotic expansion of $h_{m,n}$. Let
\begin{align*}
    \Delta_{m,0}=\left.\frac{\partial}{\partial\theta_m^{\top}}\vech {\bf{\Sigma}}_m(\theta_m)\right|_{\theta_m=\theta_{m,0}}
\end{align*}
and
\begin{align*}
    \mathbb{Y}_m(\theta_m)=-2^{-1}
    \Bigl({\bf{\Sigma}}_m(\theta_m)^{-1}-{\bf{\Sigma}}_m(\theta_{m,0})^{-1}\Bigr)\Bigl[{\bf{\Sigma}}_m(\theta_{m,0})\Bigr]-2^{-1}
    \log\frac{\det{\bf{\Sigma}}_m(\theta_m)}{\det{\bf{\Sigma}}_m(\theta_{m,0})}.
\end{align*}
Furthermore, we make the assumptions as follows.
\begin{enumerate}
    \vspace{2mm}
    \item[\bf{[C1]}]
    \begin{enumerate}
        \item There exists a constant $\chi>0$ such that
        \begin{align*}
         \mathbb{Y}_m(\theta_m)\leq -\chi\bigl|\theta_m-\theta_{m,0}\bigr|^2
        \end{align*}
        for all $\theta_m\in\Theta_m$.
        \vspace{2mm}
        \item $\rank \Delta_{m,0}=q_m$.
    \end{enumerate}
\end{enumerate}
\begin{remark}
    {\bf{[C1]}} (a) is the identifiability condition, which implies the consistency of $\hat{\theta}_{m,n}$. For the proof, see Lemma $2$ in Kusano and Uchida \cite{Kusano(2024c)}. {\bf{[C1]}} (b) 
    implies that the asymptotic variance of $\hat{\theta}_{m,n}$ is a positive definite matrix; see Lemma 35 in Kusano and Uchida \cite{Kusano(2023)}.
\end{remark}
Set 
\begin{align*}
    J_{m,n}=\Bigl\{-n^{-1}\partial^2_{\theta_m}\mathbb{H}_{m,n}(\hat{\theta}_{m,n})>0\Bigr\}.
\end{align*}
Note that ${\bf{P}}(J_{m,n})\longrightarrow 1$ as $n\longrightarrow\infty$ under ${[\bf{A}]}$ and ${[\bf{C1}]}$ since it holds from the proof of Lemma \ref{thetaasym} that $-n^{-1}\partial^2_{\theta_m}\mathbb{H}_{m,n}(\hat{\theta}_{m,n})$ converges in probability to a positive definite matrix. We have the asymptotic expansion of the marginal quasi-log likelihood $h_{m,n}$ on $J_{m,n}$ as follows.
\begin{theorem}\label{Hexpansion}
Suppose that
\begin{align}
    \log \pi_{m,n}(\hat{\theta}_{m,n})=\log\pi_{m,n}(\theta_{m,0})+o_p(1). \label{pithetahat}
\end{align}
Under $\bf{[A]}$, $\bf{[B]}$ and $\bf{[C1]}$, as $n\longrightarrow\infty$,
{\setlength{\abovedisplayskip}{8pt}
\setlength{\belowdisplayskip}{8pt}
\begin{align*}
    \begin{split}
    &\quad\ \log\left(\int_{\Theta_m}\exp{\bigl\{\mathbb{H}_{m,n}(\theta_m)\bigr\}}\pi_{m,n}(\theta_m)d\theta_m\right)\\
    &=\mathbb{H}_{m,n}(\hat{\theta}_{m,n})+2^{-1}q_m\log 2\pi-2^{-1}\log\det{\Bigl\{-\partial^2_{\theta_m}\mathbb{H}_{m,n}(\hat{\theta}_{m,n})\Bigr\}}+\log\pi_{m,n}(\hat{\theta}_{m,n})+o_p(1)
    \end{split}
\end{align*}}
on $J_{m,n}$.
\end{theorem}
Note that it follows from Theorem \ref{Hexpansion} that 
{\setlength{\abovedisplayskip}{8pt}
\setlength{\belowdisplayskip}{8pt}
\begin{align*}
    \begin{split}
    &\quad\ -2\log\left(\int_{\Theta_m}\exp{\bigl\{\mathbb{H}_{m,n}(\theta_m)\bigr\}}\pi_{m,n}(\theta_m)d\theta_m\right)\\
    &=-2\mathbb{H}_{m,n}(\hat{\theta}_{m,n})-q_m\log 2\pi+\log\det{\Bigl\{-\partial^2_{\theta_m}\mathbb{H}_{m,n}(\hat{\theta}_{m,n})\Bigr\}}-2\log\pi_{m,n}(\hat{\theta}_{m,n})+o_p(1)
    \end{split}
\end{align*}
on $J_{m,n}$ as $n\longrightarrow\infty$.} Since it holds from {[\bf{B}]} ({\rm{a}}) that
\begin{align*}
    \log\pi_{m,n}(\theta_{m,0})=O_p(1)
\end{align*}
as $n\longrightarrow\infty$, we see from (\ref{pithetahat}) that
\begin{align*}
    -q_m\log 2\pi-2\log\pi_{m,n}(\hat{\theta}_{m,n})+o_p(1)=O_p(1).
\end{align*}
Hence, ${\rm{QBIC}}_{1,n}$ is set as
\begin{align*}
    {\rm{QBIC}}^{(m)}_{1,n}=
    -2\mathbb{H}_{m,n}(\hat{\theta}_{m,n})+\log\det{n\tilde{\bf{\Gamma}}_{m,n}(\hat{\theta}_{m,n})},
\end{align*}
where
{\setlength{\abovedisplayskip}{8pt}
\setlength{\belowdisplayskip}{8pt}
\begin{align*}
    \tilde{\bf{\Gamma}}_{m,n}(\hat{\theta}_{m,n})=\left\{
    \begin{array}{ll}
    -n^{-1}\partial^2_{\theta_m}\mathbb{H}_{m,n}(\hat{\theta}_{m,n}), & (\mbox{on}\ J_{m,n}),\\
    \mathbb{I}_{q_m}, & (\mbox{on}\ J^c_{m,n}).
    \end{array}
    \right.
\end{align*}
Moreover, it follows from Lemma 6 that}
\begin{align*}
    \log\det{n\tilde{\bf{\Gamma}}_{m,n}(\hat{\theta}_{m,n})}&=q_m\log n+\log\det{\tilde{\bf{\Gamma}}_{m,n}(\hat{\theta}_{m,n})}\\
    &=q_m\log n+O_p(1)
\end{align*}
as $n\longrightarrow\infty$. Therefore, ${\rm{QBIC}}_{2,n}$ is defined as
\begin{align*}
    {\rm{QBIC}}^{(m)}_{2,n}=-2\mathbb{H}_{m,n}(\hat{\theta}_{m,n})+
    q_m\log n.
\end{align*}
Note that ${\rm{QBIC}}^{(m)}_{1,n}={\rm{QBIC}}^{(m)}_{2,n}$ on $J_{m,n}^c$.

Next, we consider the situation where the set of competing models includes some (not all) misspecified parametric models. 
Set 
{\setlength{\abovedisplayskip}{8pt}
\setlength{\belowdisplayskip}{8pt}
\begin{align*}
    \mathfrak{M}=\biggl\{m\in\{1,\cdots,M\}\ \Big|\ \mbox{There exists}\ \theta_{m,0}\in\Theta_{m}\ \mbox{such that}\ {\bf{\Sigma}}_0={\bf{\Sigma}}_{m}(\theta_{m,0}).\biggr\}
\end{align*}
and $\mathfrak{M}^{c}=\{1,\cdots,M\}\backslash \mathfrak{M}$.} Note that $\mathfrak{M}$ is the set of correctly specified models. The optimal model $m^*$ is defined as follows:
\begin{align*}
    \{m^*\}={\rm{argmin}}_{m\in\mathfrak{M}}\ {\rm{dim}}(\Theta_m).
\end{align*}
Note that we assume that the optimal model is unique. The optimal parameter $\bar{\theta}_{m}$ is set as
\begin{align*}
    \mathbb{H}_{m,0}(\bar{\theta}_{m})=\sup_{\theta_m\in\Theta_m}\mathbb{H}_{m,0}(\theta_{m}),
\end{align*}
where
\begin{align*}
    \mathbb{H}_{m,0}(\theta_m)
    &=-2^{-1}{\bf{\Sigma}}_m
    (\theta_m)^{-1}\bigl[{\bf{\Sigma}}_0\bigr]-2^{-1}
    \log\det{\bf{\Sigma}}_m(\theta_m).
\end{align*}
For $m\in\mathfrak{M}$, we note that $\bar{\theta}_m=\theta_{m,0}$. The following assumption is made.
\begin{enumerate}
    \vspace{3mm}
    \item[{\bf{[C2]}}] 
    \begin{enumerate}
        \item $\mathbb{H}_{m,0}(\theta_m)=\mathbb{H}_{m,0}(\bar{\theta}_{m})\Longrightarrow \
        \theta_m=\bar{\theta}_{m}$.
        \vspace{2mm}
        \item $-\partial^2_{\theta_m}\mathbb{H}_{m,0}(\bar{\theta}_m)>0$.
    \end{enumerate}
    \vspace{3mm}
\end{enumerate}
Furthermore, we define that Model $i$ is nested in Model $j$ when $q_i<q_j$ and there exist a matrix $F\in\mathbb{R}^{q_j\times q_i}$ with $F^{\top}F=\mathbb{I}_{q_i}$ and a constant $c\in\mathbb{R}^{q_j}$ such that $\mathbb{H}_{i,n}(\theta_{i})=\mathbb{H}_{j,n}(F\theta_{i}+c)$ for all $\theta_i\in\Theta_i$. By the following theorem, we obtain the model selection consistency of ${\rm{QBIC}}_{1,n}$ and ${\rm{QBIC}}_{2,n}$.
\begin{theorem}\label{BICcons1} Under {\bf{[A]}}, {\bf{[B]}} and {\bf{[C1]}}, as Model $m^*$ is nested in Model $m\in\mathfrak{M}\backslash\{m^*\}$, 
\begin{align*}
    {\bf{P}}\Bigl({\rm{QBIC}}^{(m^*)}_{1,n}<{\rm{QBIC}}^{(m)}_{1,n}\Bigr)
    &\longrightarrow 1,\\
    {\bf{P}}\Bigl({\rm{QBIC}}^{(m^*)}_{2,n}<{\rm{QBIC}}^{(m)}_{2,n}\Bigr)
    &\longrightarrow 1
\end{align*}
as $n\longrightarrow\infty$.
\end{theorem}
Moreover, we have the following asymptotic result.
\begin{theorem}\label{BICcons2} Under {\bf{[A]}}, {\bf{[B]}} and {\bf{[C2]}}, as $n\longrightarrow\infty$,
\begin{align*}
    {\bf{P}}\Bigl({\rm{QBIC}}^{(m^*)}_{1,n}<{\rm{QBIC}}^{(m)}_{1,n}\Bigr)
    &\longrightarrow 1,\\
    {\bf{P}}\Bigl({\rm{QBIC}}^{(m^*)}_{2,n}<{\rm{QBIC}}^{(m)}_{2,n}\Bigr)
    &\longrightarrow 1
\end{align*}
for $m\in\mathfrak{M}^c$.
\end{theorem}
Theorem \ref{BICcons2} shows that the probability of choosing the misspecified models converges to zero as $n\longrightarrow\infty$.
\section{Numerical experiments}
\subsection{True model}
The stochastic process $\mathbb{X}_{1,0,t}$ is defined as follows:
{\setlength{\abovedisplayskip}{8pt}
\setlength{\belowdisplayskip}{8pt}
\begin{align*}
    \mathbb{X}_{1,0,t}&=\begin{pmatrix}
    1 & 3 & 4 & 6
    \end{pmatrix}^{\top}
    \xi_{0,t}+\delta_{0,t},
\end{align*}
where $\{\mathbb{X}_{1,0,t}\}_{t\geq 0}$ is a four-dimensional observable vector process, $\{\xi_{0,t}\}_{t\geq 0}$ is a one-dimensional latent common factor vector process, and $\{\delta_{0,t}\}_{t\geq 0}$ is a four-dimensional latent unique factor vector process. } Assume that $\{\xi_{0,t}\}_{t\geq 0}$ and $\{\delta_{0,t}\}_{t\geq 0}$ are  the one and four-dimensional OU processes as follows:
\begin{align*}
    \quad\dd \xi_{0,t}&=-\bigl(2\xi_{0,t}-5\bigr)\dd t+3\dd W_{1,t}, \ \ 
    \xi_{0,0}=3, \\
    \quad\dd\delta_{0,t}&=-\bigl(B_{2,0}\delta_{0,t}-\mu_{2,0}\bigr)
    \dd t+{\bf{S}}_{2,0}\dd W_{2,t},\ \ 
    \delta_{0,0}=0,      
\end{align*}
where $W_{1,t}$ and $W_{2,t}$ are one and four-dimensional standard Wiener processes, and
\begin{align*}
    B_{2,0}={\rm{Diag}}\bigl(5,2,1,3\bigr)^{\top},\quad \mu_{2,0}=\bigl(4,2,1,2\bigr)^{\top},\quad
    {\bf{S}}_{2,0}={\rm{Diag}}\bigl(2,1,2,3\bigr)^{\top}.
\end{align*}
The stochastic process $\mathbb{X}_{2,0,t}$ is defined by the following factor model:
{\setlength{\abovedisplayskip}{8pt}
\setlength{\belowdisplayskip}{8pt}
\begin{align*}
    \mathbb{X}_{2,0,t}&=\begin{pmatrix}
    1 & 3 & 2 & 0 & 0 & 0\\
    0 & 0 & 0 & 1 & 2 & 4
    \end{pmatrix}^{\top}\eta_{0,t}+\varepsilon_{0,t},
\end{align*}
where $\{\mathbb{X}_{2,0,t}\}_{t\geq 0}$ is a six-dimensional observable vector process, $\{\eta_{0,t}\}_{t\geq 0}$ is a two-dimensional latent common factor vector process, and $\{\varepsilon_{0,t}\}_{t\geq 0}$ is a six-dimensional latent unique factor vector process, respectively.} Suppose that $\{\varepsilon_{0,t}\}_{t\geq 0}$  satisfies the following  stochastic differential equation:
\begin{align*}
    \quad\dd\varepsilon_{0,t}&=-\bigl(B_{3,0}\varepsilon_{0,t}-\mu_{3,0}\bigr)\dd t+{\bf{S}}_{3,0}\dd W_{3,t}, \ \ 
     \varepsilon_{0,0}=0,     
\end{align*}
where $W_{3,t}$ is a six-dimensional standard Wiener process, and
\begin{align*}
    B_{3,0}={\rm{Diag}}\bigl(1,5,2,3,2,2\bigr)^{\top},\quad \mu_{3,0}=\bigl(2,1,3,2,1,4\bigr)^{\top},\quad {\bf{S}}_{3,0}={\rm{Diag}}\bigl(5,1,2,1,3,2\bigr)^{\top}.
\end{align*}
Furthermore, the relationship between $\eta_{0,t}$ and  $\xi_{0,t}$ is expressed as follows:
{\setlength{\abovedisplayskip}{8pt}
\setlength{\belowdisplayskip}{8pt}
\begin{align*}
    \eta_{0,t}=\begin{pmatrix}
    3 \\
    2
    \end{pmatrix}\xi_{0,t}+\zeta_{0,t},
\end{align*}
where $\{\zeta_{0,t}\}_{t\geq 0}$ is a two-dimensional latent unique factor vector process defined by the following two-dimensional OU process:}
\begin{align*}
    \quad\dd\zeta_{0,t}=-\left\{\begin{pmatrix}
    3 & 0 \\
    0 & 2
    \end{pmatrix}\zeta_{0,t}-\begin{pmatrix}
    1\\
    2
    \end{pmatrix}\right\}\dd t+\begin{pmatrix}
    3 & 0\\
    0 & 1
    \end{pmatrix}\dd W_{4,t},\ \ \zeta_{0,0}=0,
\end{align*}
where $W_{4,t}$ is a two-dimensional standard Wiener process. Figure \ref{truemodel} shows the path diagram of the true model at time $t$.
\begin{figure}[t]
    \includegraphics[width=0.9\columnwidth]{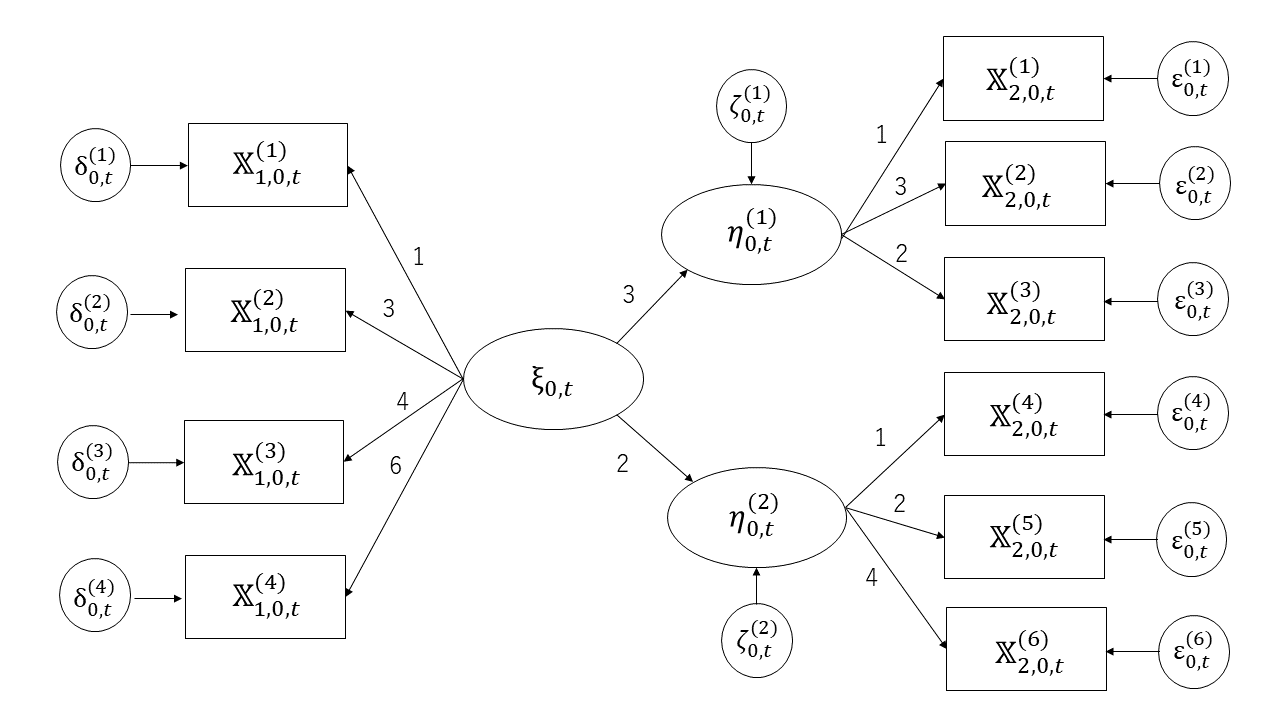}
    \caption{The path diagram of the true model at time $t$.}\label{truemodel}
\end{figure}
\subsection{Parametric models}
\subsubsection{Model 1}
Let $p_1=4$, $p_2=6$, $k_{1}=1$, $k_{2}=2$ and $q_1=22$. Suppose
{\setlength{\abovedisplayskip}{8pt}
\setlength{\belowdisplayskip}{8pt}
\begin{align*}
    {\bf{\Lambda}}^{\theta}_{1,x_1}&=\begin{pmatrix}
    1 & \theta^{(1)}_1 & 
    \theta^{(2)}_1 &  \theta^{(3)}_1
    \end{pmatrix}^{\top},\quad
    {\bf{\Lambda}}^{\theta}_{1,x_2}=\begin{pmatrix}
    1 &
    \theta_1^{(4)} &
    \theta_1^{(5)} &
    0 & 0 & 0\\
    0 & 0 & 0 & 1 & \theta_1^{(6)} &
    \theta_1^{(7)}
    \end{pmatrix}^{\top}
\end{align*}
and}
\begin{align*}
    {\bf{\Gamma}}^{\theta}_{1}=\begin{pmatrix}
    \theta_1^{(8)} &
    \theta_1^{(9)}
    \end{pmatrix}^{\top},\quad 
    {\bf{\Psi}}_1^{\theta}=\mathbb{I}_2,
\end{align*}
where $\theta_1^{(i)}$ for $i=1,\cdots,9$ are not zero. 
It is assumed that ${\bf{S}}^{\theta}_{1,1}$ and ${\bf{S}}^{\theta}_{2,1}$ satisfy 
\begin{align*}
    {\bf{\Sigma}}^{\theta}_{1,\xi\xi}=({\bf{S}}^{\theta}_{1,1})^2=\theta_1^{(10)}>0
\end{align*}
and 
\begin{align*}
    &{\bf{\Sigma}}^{\theta}_{1,\delta\delta}={\bf{S}}^{\theta}_{2,1}{\bf{S}}^{\theta\top}_{2,1}={\rm{Diag}}\Bigl(\theta_1^{(11)},\theta_1^{(12)},
    \theta_1^{(13)},\theta_1^{(14)}
    \Bigr)^{\top}\in\mathcal{M}_4^{+}.
\end{align*}
Furthermore, we suppose that ${\bf{S}}^{\theta}_{3,1}$ and ${\bf{S}}^{\theta}_{4,1}$ satisfy
\begin{align*}
    &\qquad\qquad\qquad{\bf{\Sigma}}^{\theta}_{1,\varepsilon\varepsilon}={\bf{S}}^{\theta}_{3,1}{\bf{S}}^{\theta\top}_{3,1}={\rm{Diag}}\Bigl(\theta_1^{(15)},\theta_1^{(16)},
    \theta_1^{(17)},\theta_1^{(18)},
    \theta_1^{(19)},\theta_1^{(20)}
    \Bigr)^{\top}\in\mathcal{M}_6^{+}
\end{align*}
and 
\begin{align*}
    {\bf{\Sigma}}^{\theta}_{1,\zeta\zeta}={\bf{S}}^{\theta}_{4,1}{\bf{S}}^{\theta\top}_{4,1}={\rm{Diag}}\Bigl(\theta_1^{(21)},\theta_1^{(22)}\Bigr)^{\top}\in\mathcal{M}_2^{+}.
\end{align*}
Let 
\begin{align*}
    \theta_{1,0}=\Bigl(3,4,6,3,2,2,4,3,2,9,4,1,4,9,25,1,4,1,9,4,9,1\Bigr)^{\top}.
\end{align*}
Since ${\bf{\Sigma}}_0={\bf{\Sigma}}_1(\theta_{1,0})$,  Model $1$ is a correctly specified model. Note that Model $1$ satisfies the identifiability condition {\bf{[C1]}} (a). For the check of the identifiability condition, see Appendix \ref{idensec}.
Figure \ref{Model1} shows the path diagram of Model $1$ at time $t$.
\begin{figure}[h]
    \includegraphics[width=0.9\columnwidth]{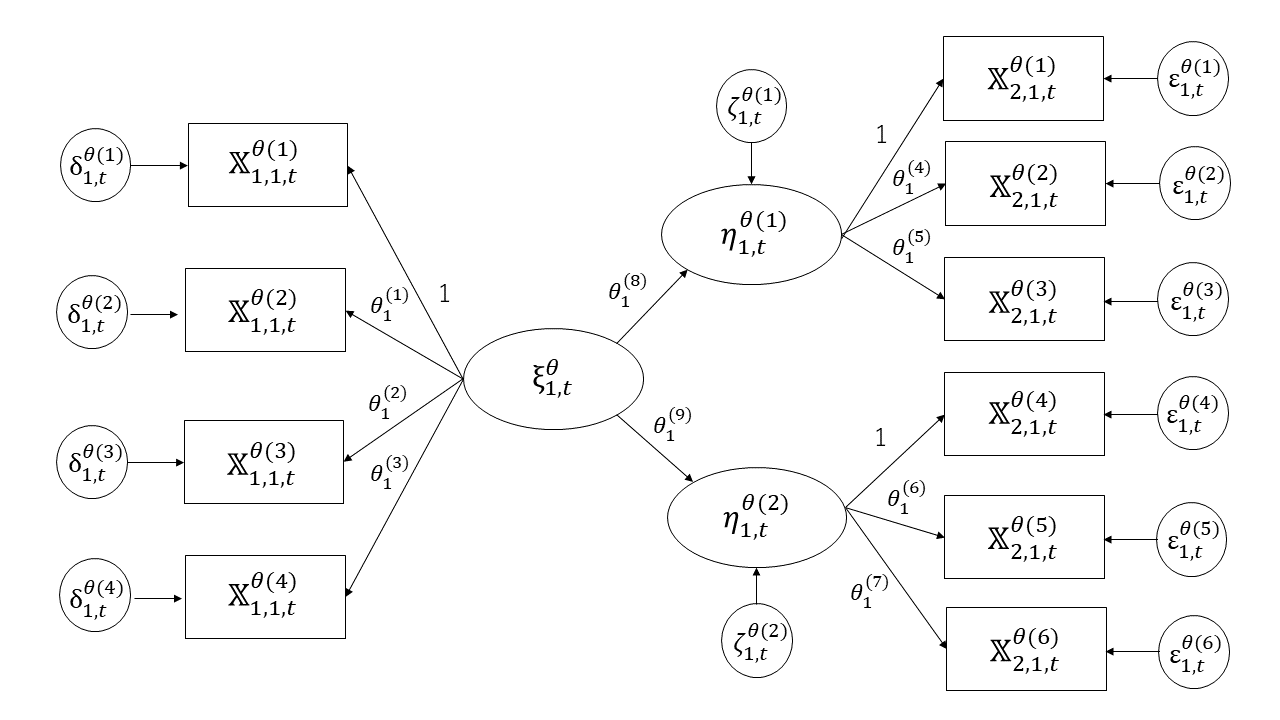}
    \caption{The path diagram of Model 1.} \label{Model1}
\end{figure}
\subsubsection{Model 2}
Set $p_1=4$, $p_2=6$, $k_{1}=1$, $k_{2}=2$ and $q_2=23$. Assume
{\setlength{\abovedisplayskip}{8pt}
\setlength{\belowdisplayskip}{8pt}
\begin{align*}
    {\bf{\Lambda}}^{\theta}_{2,x_1}&=\begin{pmatrix}
    1 & \theta^{(1)}_2 & 
    \theta^{(2)}_2 &  \theta^{(3)}_2
    \end{pmatrix}^{\top},\quad
    {\bf{\Lambda}}^{\theta}_{2,x_2}=\begin{pmatrix}
    1 &
    \theta_2^{(4)} &
    \theta_2^{(5)} &
    0 & 0 & 0\\
    0 & 0 & \theta_2^{(6)} & 1 & \theta_2^{(7)} &
    \theta_2^{(8)}
    \end{pmatrix}^{\top},
\end{align*}
and}
\begin{align*}
    {\bf{\Gamma}}^{\theta}_{2}=\begin{pmatrix}
    \theta_2^{(9)} &
    \theta_2^{(10)}
    \end{pmatrix}^{\top},\quad 
    {\bf{\Psi}}_2^{\theta}=\mathbb{I}_2,
\end{align*}
where $\theta_2^{(i)}$ for $i=1,\cdots,5$ and $i=7,\cdots,10$ are not zero. ${\bf{S}}^{\theta}_{1,2}$ and ${\bf{S}}^{\theta}_{2,2}$ are supposed to satisfy 
\begin{align*}
    {\bf{\Sigma}}^{\theta}_{2,\xi\xi}=({\bf{S}}^{\theta}_{1,2})^2=\theta_1^{(11)}>0
\end{align*}
and
\begin{align*}
    {\bf{\Sigma}}^{\theta}_{2,\delta\delta}&={\bf{S}}^{\theta}_{2,2}{\bf{S}}^{\theta\top}_{2,2}={\rm{Diag}}\Bigl(\theta_2^{(12)},\theta_2^{(13)},
    \theta_2^{(14)},\theta_2^{(15)} \Bigr)^{\top}\in\mathcal{M}_4^{+}.
\end{align*}
Moreover, we assume that ${\bf{S}}^{\theta}_{3,2}$ and ${\bf{S}}^{\theta}_{4,2}$ satisfy
\begin{align*}
    &\qquad\qquad\qquad{\bf{\Sigma}}^{\theta}_{2,\varepsilon\varepsilon}={\bf{S}}^{\theta}_{3,2}{\bf{S}}^{\theta\top}_{3,2}={\rm{Diag}}\Bigl(\theta_2^{(16)},
    \theta_2^{(17)},\theta_2^{(18)},
    \theta_2^{(19)},\theta_2^{(20)},\theta_2^{(21)}
    \Bigr)^{\top}\in\mathcal{M}_6^{+}
\end{align*}
and 
\begin{align*}
    {\bf{\Sigma}}^{\theta}_{2,\zeta\zeta}={\bf{S}}^{\theta}_{4,2}{\bf{S}}^{\theta\top}_{4,2}={\rm{Diag}}\Bigl(\theta_1^{(22)},\theta_1^{(23)}\Bigr)^{\top}\in\mathcal{M}_2^{+}.
\end{align*}
Let 
\begin{align*}
    \theta_{2,0}=\Bigl(3,4,6,3,2,0,2,4,3,2,9,4,1,4,9,25,1,4,1,9,4,9,1\Bigr)^{\top}.
\end{align*}
It holds that ${\bf{\Sigma}}_0={\bf{\Sigma}}_2(\theta_{2,0})$, so that Model $2$ is a correctly specified model. In a similar way to
Model $1$, we can check the identifiability condition of Model $2$. Figure \ref{Model2} shows the path diagram of Model $2$ at time $t$.
\begin{figure}[h]
    \includegraphics[width=0.9\columnwidth]{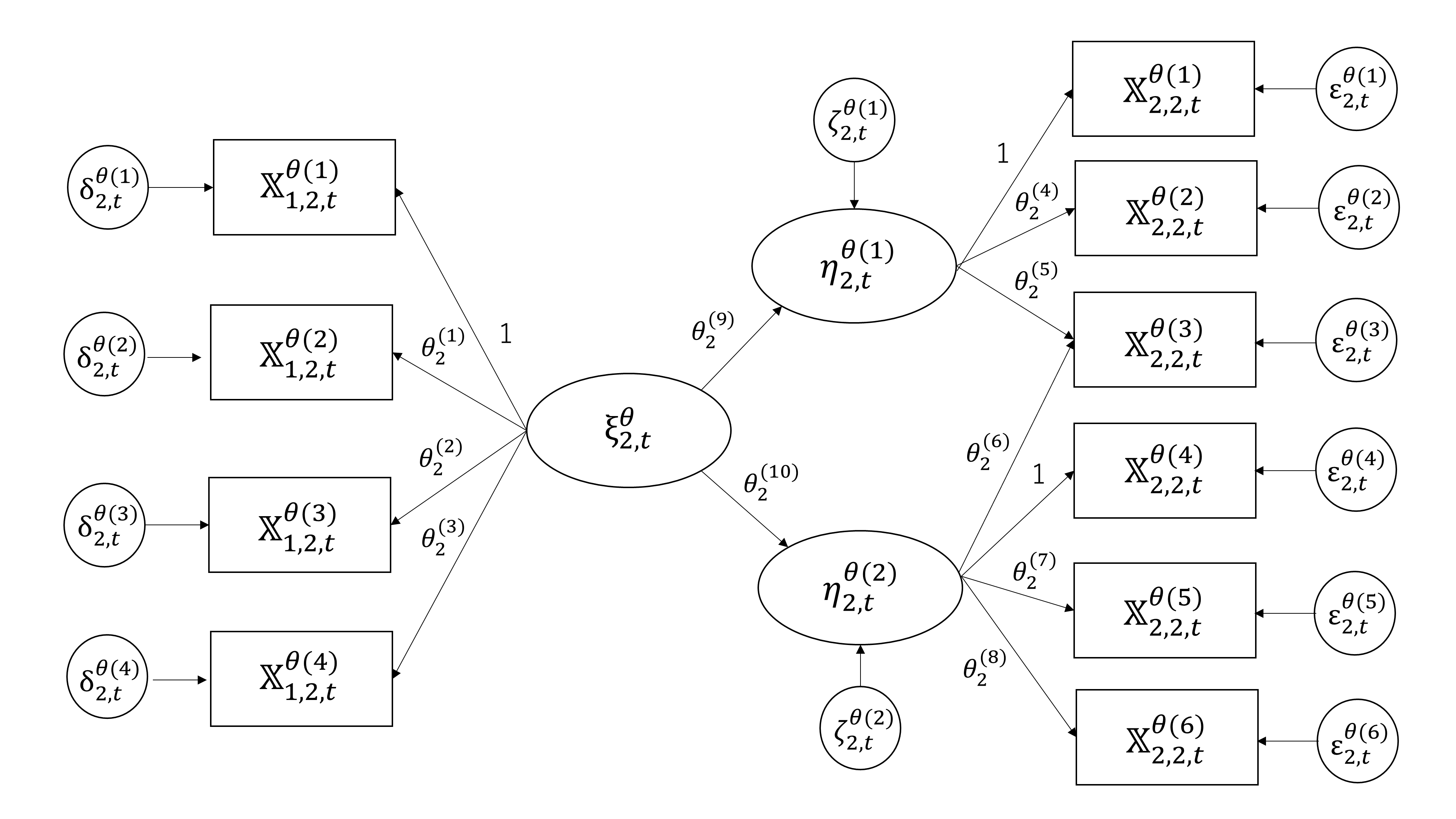}
    \caption{The path diagram of Model 2.} \label{Model2}
\end{figure}
\subsubsection{Model 3} 
Let $p_1=4$, $p_2=6$, $k_{1}=1$, $k_{2}=1$ and $q_3=21$. Suppose 
\begin{align*}
    {\bf{\Lambda}}^{\theta}_{3,x_1}&=\begin{pmatrix}
    1 &
    \theta^{(1)}_3 &
    \theta^{(2)}_3 &
    \theta^{(3)}_3 
    \end{pmatrix}^{\top}
\end{align*}
and
\begin{align*}
    {\bf{\Lambda}}^{\theta}_{3,x_2}=\begin{pmatrix}
    1 &
    \theta_3^{(4)} &
    \theta^{(5)}_3 &
    \theta^{(6)}_3 &
    \theta^{(7)}_3 &
    \theta^{(8)}_3
    \end{pmatrix}^{\top},\quad
    {\bf{\Gamma}}^{\theta}_{3}=
    \theta_3^{(9)},
\end{align*}
where $\theta^{(i)}_3$ for $i=1,\cdots,9$ are not zero. It is assumed that ${\bf{S}}^{\theta}_{1,3}$ and ${\bf{S}}^{\theta}_{2,3}$ satisfy 
\begin{align*}
    {\bf{\Sigma}}^{\theta}_{3,\xi\xi}=({\bf{S}}^{\theta}_{1,3})^2=\theta_3^{(10)}>0
\end{align*}
and
\begin{align*}
    {\bf{\Sigma}}^{\theta}_{3,\delta\delta}&={\bf{S}}^{\theta}_{2,3}{\bf{S}}^{\theta\top}_{2,3}
    ={\rm{Diag}}\Bigl(\theta_3^{(11)},\theta_3^{(12)},
    \theta_3^{(13)},
    \theta_3^{(14)}
    \Bigr)^{\top}\in\mathcal{M}_4^{+}.
\end{align*}
Furthermore, we suppose that ${\bf{S}}^{\theta}_{3,3}$ and ${\bf{S}}^{\theta}_{4,3}$ satisfy
\begin{align*}
    {\bf{\Sigma}}^{\theta}_{3,\varepsilon\varepsilon}&={\bf{S}}^{\theta}_{3,3}{\bf{S}}^{\theta\top}_{3,3}=
    {\rm{Diag}}\Bigl(\theta_3^{(15)},\theta_3^{(16)},
    \theta_3^{(17)},
    \theta_3^{(18)},
    \theta_3^{(19)},
    \theta_3^{(20)}
    \Bigr)^{\top}
    \in\mathcal{M}_6^{+}
\end{align*}
and
\begin{align*}
    {\bf{\Sigma}}^{\theta}_{3,\zeta\zeta}=({\bf{S}}^{\theta}_{4,3})^2
    =\theta_3^{(21)}>0.
\end{align*}
Since we have ${\bf{\Sigma}}_0\neq{\bf{\Sigma}}_3(\theta_{3})$ for any $\theta_3\in\Theta_3$, Model $3$ is a misspecified model. Figure \ref{Model3} shows the path diagram of Model $3$ at time $t$.
\begin{figure}[h]
    \includegraphics[width=0.9\columnwidth]{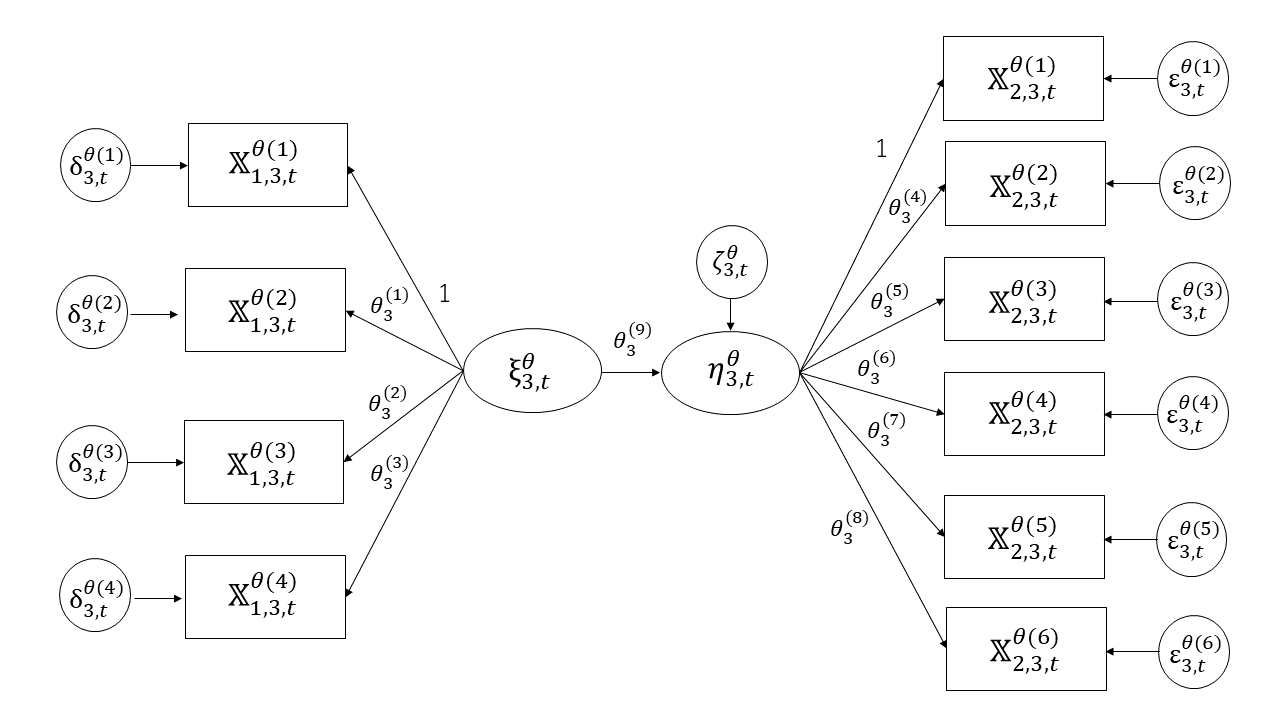}
    \caption{The path diagram of Model 3.} \label{Model3}
\end{figure}
\subsection{Simulation results}
In the simulation, we use optim() with the BFGS method in R language. 
The true value was used as the initial value for optimisation. $10000$ independent sample paths are generated from the true model. Let $T=1$ and we consider the case where $n=10^2, 10^3, 10^4, 10^5$. To see the differences in model selection, we use 
\begin{align*}
    {\rm{QAIC}}_n^{(m)}=-2\mathbb{H}_{m,n}(\hat{\theta}_{m,n})+2q_m.
\end{align*}
See Kusano and Uchida \cite{Kusano(2024c)} for the details of QAIC for SEM with latent variables for diffusion processes. 

Table \ref{table} shows the number of models selected by ${\rm{QBIC}}_{1}$, ${\rm{QBIC}}_{2}$  and QAIC. For any cases, ${\rm{QBIC}}_{1}$ and ${\rm{QBIC}}_{2}$ choose the optimal model (Model $1$), which implies that Theorem \ref{BICcons1} seems to be correct in this example. 
In all cases, ${\rm{QBIC}}_{2}$ is more likely to select the optimal model than ${\rm{QBIC}}_{1}$ especially as sample size is small. On the other hand, $\rm{QAIC}$ selects the over-fitted model (Model $2$) with significant probability even for a large sample size. This implies that $\rm{QAIC}$ does not have the model selection consistency. 
Furthermore, in all cases, the misspecified model (Model $3$) is not selected by ${\rm{QBIC}}_{1}$, ${\rm{QBIC}}_{2}$ and $\rm{QAIC}$. It seems that Theorem $2$ in Kusano and Uchida \cite{Kusano(2024c)} and Theorem \ref{BICcons2} hold true for this example. 
\begin{table}[h]
    \vspace{5mm}
    \begin{tabular}{ccccc}
    $n=10^2$ & Model 1$^*$ & Model 2 & Model 3 \\\hline
    ${\rm{QBIC}}_{1}$ & 9469 & 531 & 0 \\
    ${\rm{QBIC}}_{2}$ & 9644 & 356 & 0\\
    QAIC & 8327 & 1673 & 0
    \end{tabular}\qquad
    \begin{tabular}{ccccc}
    $n=10^3$ & Model 1$^*$  & Model 2 & Model 3 \\\hline
    ${\rm{QBIC}}_{1}$ & 9854 & 146 & 0 \\
    ${\rm{QBIC}}_{2}$ & 9911 & 89 & 0\\
    QAIC & 8465 & 1535 & 0
    \end{tabular}\\
    \vspace{4mm}
    \begin{tabular}{ccccc}
    $n=10^4$ & Model 1$^*$  & Model 2 & Model 3 \\\hline
    ${\rm{QBIC}}_{1}$ & 9956 & 44 & 0 \\
    ${\rm{QBIC}}_{2}$ & 9972 & 28 & 0\\
    QAIC & 8443 & 1557 & 0
    \end{tabular}\qquad
    \begin{tabular}{ccccc}
    $n=10^5$ & Model 1$^*$  & Model 2 & Model 3 \\\hline
    ${\rm{QBIC}}_{1}$ & 9991 & 9 & 0\\
    ${\rm{QBIC}}_{2}$ & 9995 & 5 & 0\\
    QAIC & 8400 & 1600 & 0
    \end{tabular}
    \caption{The number of models selected by ${\rm{QBIC}}_{1}$, ${\rm{QBIC}}_{2}$ and QAIC.}\label{table}
\end{table}
\section{Proof}
In this section, the model index "$m$" may be omitted. Let ${\bf{\Gamma}}_0=\Delta_0^{\top}{\bf{W}}_0^{-1}\Delta_0$, where 
\begin{align*}
    {\bf{W}}_0=2\mathbb{D}_p^{+}\bigl({\bf{\Sigma}}_0\otimes {\bf{\Sigma}}_0\bigr)\mathbb{D}^{+\top}_p.
\end{align*}
Set the random field $\mathbb{Z}_n$ as 
\begin{align*}
    \mathbb{Z}_n(u)=\mathbb{Z}_n(u;\theta_0)=\exp\Bigl\{\mathbb{H}_n\bigl(\theta_0+n^{-\frac{1}{2}}u\bigr)-\mathbb{H}_n(\theta_0)\Bigr\}
\end{align*}
for $u\in\mathbb{U}_n(\theta_0)$ and $\mathbb{Z}_n(u)=0$ for $u\in\mathbb{U}_n(\theta_0)^c$.
Furthermore, we define
\begin{align*}
    \mathbb{V}_n(r)=\mathbb{U}_n(\theta_0)\cap\bigl\{u\in\mathbb{R}^{q}:|u|\geq r\bigr\}=\Bigl\{u\in\mathbb{U}_n(\theta_0):|u|\geq r\Bigr\}
\end{align*}
for $r>0$. 
\begin{lemma}\label{A1}
Under {\bf{[A]}}, as $n\longrightarrow\infty$,
\begin{align}
    \frac{1}{\sqrt{n}}\mathbb{H}_n(\theta_0)=O_p(1),\quad -\frac{1}{n}\partial^2_{\theta}\mathbb{H}_n(\theta_0)\stackrel{p}{\longrightarrow}{\bf{\Gamma}}_0>0 \label{gamma}
\end{align}
and
\begin{align}
    \sup_{\theta\in\Theta}\biggl|\frac{1}{n\sqrt{n}}\partial_{\theta^{(j_1)}}\partial_{\theta^{(j_2)}}
    \partial_{\theta^{(j_3)}}\mathbb{H}_n(\theta)\biggr|=o_p(1) \label{supH}
\end{align}
for $j_1,j_2,j_3=1,\cdots,q$.
\end{lemma}
\begin{proof}
See Lemma $1$ in Kusano and Uchida \cite{Kusano(2024c)} for the proof of (\ref{gamma}). Note that
\begin{align*}
    n^{-1}\partial_{\theta^{(j_1)}}\partial_{\theta^{(j_2)}}
    \partial_{\theta^{(j_3)}}\mathbb{H}_n(\theta)&=-\frac{1}{2nh_n}\sum_{i=1}^n\Bigl(\partial_{\theta^{(j_1)}}\partial_{\theta^{(j_2)}}
    \partial_{\theta^{(j_3)}}{\bf{\Sigma}}(\theta)^{-1}\Bigr)\Bigl[\bigl(\Delta_i\mathbb{X}\bigr)^{\otimes 2}\Bigr]\\
    &\qquad\qquad\qquad\qquad\qquad\qquad-2^{-1}\partial_{\theta^{(j_1)}}\partial_{\theta^{(j_2)}}
    \partial_{\theta^{(j_3)}}\log\det{\bf{\Sigma}}(\theta)\\
    &=-2^{-1}\Bigl(\partial_{\theta^{(j_1)}}\partial_{\theta^{(j_2)}}
    \partial_{\theta^{(j_3)}}{\bf{\Sigma}}(\theta)^{-1}\Bigr)\bigl[\mathbb{Q}_{\mathbb{XX}}\bigr]\\
    &\qquad\qquad\qquad\qquad\qquad\qquad-2^{-1}\partial_{\theta^{(j_1)}}\partial_{\theta^{(j_2)}}
    \partial_{\theta^{(j_3)}}\log\det{\bf{\Sigma}}(\theta)
\end{align*}
for $j_1,j_2,j_3=1,\cdots,q$. Since it holds from Theorem $1$ in Kusano and Uchida \cite{Kusano(JJSD)} that $\mathbb{Q}_{\mathbb{XX}}\stackrel{p}{\longrightarrow}{\bf{\Sigma}}_0$, the continuous mapping theorem shows $|\mathbb{Q}_{\mathbb{XX}}|\stackrel{p}{\longrightarrow}|{\bf{\Sigma}}_0|$.
Hence, it follows from the compactness of $\Theta$ that
{\setlength{\abovedisplayskip}{8pt}
\setlength{\belowdisplayskip}{8pt}
\begin{align*}
    &\quad\ \sup_{\theta\in\Theta}\biggl|\frac{1}{n\sqrt{n}}\partial_{\theta^{(j_1)}}\partial_{\theta^{(j_2)}}
    \partial_{\theta^{(j_3)}}\mathbb{H}_n(\theta)\biggr|\\
    &\leq \frac{1}{\sqrt{n}}\sup_{\theta\in\Theta}\Bigl|
    \Bigl(\partial_{\theta^{(j_1)}}\partial_{\theta^{(j_2)}}
    \partial_{\theta^{(j_3)}}{\bf{\Sigma}}(\theta)^{-1}\Bigr)\bigl[\mathbb{Q}_{\mathbb{XX}}\bigr]
    \Bigr|+\frac{1}{\sqrt{n}}\sup_{\theta\in\Theta}\Bigl|\partial_{\theta^{(j_1)}}\partial_{\theta^{(j_2)}}
    \partial_{\theta^{(j_3)}}\log\det{\bf{\Sigma}}(\theta)\Bigr|\\
    &\leq \frac{1}{\sqrt{n}}\bigl|\mathbb{Q}_{\mathbb{XX}}\bigr|\sup_{\theta\in\Theta}\Bigl|
    \partial_{\theta^{(j_1)}}\partial_{\theta^{(j_2)}}
    \partial_{\theta^{(j_3)}}{\bf{\Sigma}}(\theta)^{-1}
    \Bigr|+\frac{1}{\sqrt{n}}\sup_{\theta\in\Theta}
    \Bigl|\partial_{\theta^{(j_1)}}\partial_{\theta^{(j_2)}}
    \partial_{\theta^{(j_3)}}\log\det{\bf{\Sigma}}(\theta)\Bigr|
    \stackrel{p}{\longrightarrow} 0
\end{align*}}
for $j_1,j_2,j_3=1,\cdots,q$, which yields (\ref{supH}).
\end{proof}
\begin{lemma}\label{thetahat}
Under {\bf{[A]}} and {\bf{[C1]}} (a),
\begin{align*}
    \hat{\theta}_n\stackrel{p}{\longrightarrow}\theta_0
\end{align*}
as $n\longrightarrow\infty$.
\end{lemma}
\begin{proof}
    See Lemma $2$ in Kusano and Uchida \cite{Kusano(2024c)}.
\end{proof}
\begin{lemma}\label{Zine}
Under {\bf{[A]}} and {\bf{[C1]}}, for all $L>0$, there exists $C_L>0$ such that
\begin{align*}
    \sup_{n\in\mathbb{N}}{\bf{P}}\left(\sup_{u\in {\mathbb{V}}_n(r)}\mathbb{Z}_n(u)\geq e^{-r}\right)\leq\frac{C_L}{r^L}
\end{align*}
for all $r>0$.
\end{lemma}
\begin{proof}
    See Lemma 8 in Kusano and Uchida \cite{Kusano(2024c)}.
\end{proof}
Eguchi and Masuda \cite{Eguchi(2018)} stated only that the following Lemma \ref{A3} can be proven using the proof of Theorem 6 in Yoshida \cite{Yoshida(2011)}. However, we provide the detailed proof of the following Lemma \ref{A3} since we consider the proof to be non-trivial.
\begin{lemma}\label{A3}
Under {\bf{[A]}} and {\bf{[C1]}}, for all $\varepsilon>0$, there exists $r>0$ such that 
\begin{align*}
    \sup_{n\in\mathbb{N}}{\bf{P}}\Biggl(\int_{|u|\geq r}\mathbb{Z}_n(u) 
    \dd u> \varepsilon\Biggr)<\varepsilon.
\end{align*}
\end{lemma}
\begin{proof}
Note that
\begin{align}
    \int_{\mathbb{U}_n(\theta_0)\cap\{|u|\geq r\}}\mathbb{Z}_n(u) \dd u=\int_{|u|\geq r}\mathbb{Z}_n(u) \dd u  \label{ZU}
\end{align}
since $\mathbb{Z}_n(u)=0$ for $u\in\mathbb{U}_n(\theta_0)^c$.
Let us consider the following decomposition:
\begin{align}
    \begin{split}
    &\quad\ {\bf{P}}\Biggl(\int_{\mathbb{U}_n(\theta_0)\cap\{|u|\geq r\}}\mathbb{Z}_n(u) \dd u >\varepsilon\Biggr)\\
    &={\bf{P}}\Biggl(\biggl\{\int_{\mathbb{U}_n(\theta_0)\cap\{|u|\geq r\}}\mathbb{Z}_n(u) \dd u >\varepsilon\biggr\}\cap\bigcup_{k=0}^{\infty}\biggl\{\sup_{u\in \mathbb{U}_n(\theta_0)\cap\{|u|\geq r+k\}}\mathbb{Z}_n(u)\geq e^{-(r+k)}\biggr\}\Biggr)\\
    &\qquad\qquad+{\bf{P}}\Biggl(\biggl\{\int_{\mathbb{U}_n(\theta_0)\cap\{|u|\geq r\}}\mathbb{Z}_n(u) \dd u >\varepsilon\biggr\}\cap\bigcap_{k=0}^{\infty}\biggl\{\sup_{u\in\mathbb{U}_n(\theta_0)\cap\{|u|\geq r+k\}}\mathbb{Z}_n(u)< e^{-(r+k)}\biggr\}\Biggr)
\end{split}\label{dec}
\end{align}
for all $r>0$. First, we consider the first term on the right-hand side of (\ref{dec}).
For all $L>0$, it holds from Lemma \ref{Zine} that there exists $C_L>0$ such that 
\begin{align*}
    {\bf{P}}\Biggl(\sup_{u\in \mathbb{U}_n(\theta_0)\cap\{|u|\geq r+k\}}\mathbb{Z}_n(u)\geq e^{-(r+k)}\Biggr)\leq \frac{C_L}{(r+k)^L}
\end{align*}
for all $r>0$ and $k\in\mathbb{N}\cup\{0\}$, so that one gets
\begin{align*}
\begin{split}
    &\quad\ {\bf{P}}\Biggl(\biggl\{\int_{\mathbb{U}_n(\theta_0)\cap\{|u|\geq r\}}\mathbb{Z}_n(u) \dd u >\varepsilon\biggr\}\cap\bigcup_{k=0}^{\infty}\biggl\{\sup_{u\in \mathbb{U}_n(\theta_0)\cap\{|u|\geq r+k\}}\mathbb{Z}_n(u)\geq e^{-(r+k)}\biggr\}\Biggr)\\
    &\leq {\bf{P}}\Biggl(\bigcup_{k=0}^{\infty}\biggl\{\sup_{{u\in \mathbb{U}_n(\theta_0)\cap\{|u|\geq r+k\}}}\mathbb{Z}_n(u)\geq e^{-(r+k)}\biggr\}\Biggr)\\
    &\leq \sum_{k=0}^{\infty}{\bf{P}}\Biggl(\sup_{{u\in \mathbb{U}_n(\theta_0)\cap\{|u|\geq r+k\}}}\mathbb{Z}_n(u)\geq e^{-(r+k)}\Biggr)\\
    &\leq C_L\sum_{k=0}^{\infty}\frac{1}{(r+k)^{L}}.
\end{split}
\end{align*}
Furthermore, we see
\begin{align*}
    \sum_{k=0}^{\infty}\frac{1}{(r+k)^{L}}&=\frac{1}{r^L}+\sum_{k=1}^{\infty}\frac{1}{(r+k)^{L}}
    \leq \frac{1}{r^L}+\int_{0}^{\infty}\frac{1}{(r+x)^{L}}\dd x=\frac{1}{r^L}+\frac{1}{(L-1)r^{L-1}},
\end{align*}
which yields 
\begin{align}
\begin{split}
    &\quad\ {\bf{P}}\Biggl(\biggl\{\int_{\mathbb{U}_n(\theta_0)\cap\{|u|\geq r\}}\mathbb{Z}_n(u) \dd u >\varepsilon\biggr\}\cap\bigcup_{k=0}^{\infty}\biggl\{\sup_{u\in \mathbb{U}_n(\theta_0)\cap\{|u|\geq r+k\}}\mathbb{Z}_n(u)\geq e^{-(r+k)}\biggr\}\Biggr)\\
    &\leq C_L\biggl\{\frac{1}{r^L}+\frac{1}{(L-1)r^{L-1}}\biggr\}.
\end{split}\label{PZ1}
\end{align}
Next, we consider the second term on the right-hand side of (\ref{dec}). Note that
\begin{align*}
    &\quad\ {\bf{P}}\Biggl(\biggl\{\int_{\mathbb{U}_n(\theta_0)\cap\{|u|\geq r\}}\mathbb{Z}_n(u) \dd u >\varepsilon\biggr\}\cap\bigcap_{k=0}^{\infty}\biggl\{\sup_{u\in\mathbb{U}_n(\theta_0)\cap\{|u|\geq r+k\}}\mathbb{Z}_n(u)< e^{-(r+k)}\biggr\}\Biggr)\\
    &={\bf{P}}\Biggl(\biggl\{\sum_{k=0}^{\infty}\int_{\mathbb{U}_n(\theta_0)\cap\{r+k \leq |u|\leq r+k+1\}}\mathbb{Z}_n(u) \dd u>\varepsilon\biggr\}\cap\bigcap_{k=0}^{\infty}\biggl\{\sup_{u\in\mathbb{U}_n(\theta_0)\cap\{|u|\geq r+k\}}\mathbb{Z}_n(u)< e^{-(r+k)}\biggr\}\Biggr)\\
    &\leq {\bf{P}}\Biggl(e^{-r}\sum_{k=0}^{\infty}e^{-k}\int_{\mathbb{U}_n(\theta_0)\cap\{r+k\leq |u|\leq r+k+1\}} \dd u >\varepsilon\Biggr)
\end{align*}
for all $r>0$. It is shown that
\begin{align}
\begin{split}
    \int_{\mathbb{U}_n(\theta_0)\cap\{r+k\leq |u|\leq r+k+1\}} \dd u &\leq \int_{r+k\leq |u|\leq r+k+1} \dd u \\
    &=\int_{0}^{2\pi}\int_{0}^{\pi}\cdots\int_{0}^{\pi}\int_{r+k}^{r+k+1}
    s^{q-1}\prod_{j=1}^{q-2}\sin^{q-j-1}\phi_j \dd s\dd\phi_1\cdots \dd\phi_{q-2}\dd\phi_{q-1}\\
    &\leq \int_{0}^{2\pi}\int_{0}^{\pi}\cdots\int_{0}^{\pi}\int_{r+k}^{r+k+1}
    s^{q-1}\dd s\dd\phi_1\cdots \dd\phi_{q-2}\dd\phi_{q-1}\\
    &\leq C_q\Bigl\{(r+k+1)^{q}-(r+k)^{q}\Bigr\}
\end{split}\label{intu1}
\end{align}
by using the polar coordinates transformation:
\begin{align*}
    u^{(1)}=s\cos\phi_1,\quad
    u^{(i)}=s\cos\phi_{i}\prod_{j=1}^{i-1}\sin\phi_{j}\ (i=2,\cdots, q-1),\quad u^{(q)}=s\prod_{j=1}^{q-1}\sin\phi_{j},
\end{align*}
where $s\geq 0$, $0\leq\phi_i\leq\pi$ for $i=1,\cdots,q-2$ and $0\leq\phi_{q-1}\leq2\pi$.
Moreover, we see
\begin{align*}
    (r+k+1)^{q}-(r+k)^{q}&=\sum_{i=0}^{q}
    \frac{q!}{i!(q-i)!}
    (r+k)^{i}-(r+k)^{q}\\
    &=\sum_{i=0}^{q-1}\frac{q!}{i!(q-i)!}
    (r+k)^{i}\\
    &=\sum_{i=0}^{q-1}\frac{q!}{i!(q-i)!}
    \Biggl\{\sum_{j=0}^i \frac{i!}{j!(i-j)!}r^{j}k^{i-j}\Biggr\}\\
    &=\sum_{i=0}^{q-1}\sum_{j=0}^{i}\frac{q!}{j!(i-j)!(q-i)!}r^{j}k^{i-j}\\
    &=\sum_{i=1}^{q-1}\sum_{j=0}^{i-1}\frac{q!}{j!(i-j)!(q-i)!}r^{j}k^{i-j}+\sum_{i=0}^{q-1}\frac{q!}{i!(q-i)!}r^{i}
\end{align*}
for all $r>0$ and $k\in\mathbb{N}\cup\{0\}$, so that it holds from (\ref{intu1}) that
\begin{align}
\begin{split}
   &\quad\ e^{-r}\sum_{k=0}^{\infty}e^{-k}\int_{\mathbb{U}_n(\theta_0)\cap\{r+k\leq |u|\leq r+k+1\}} \dd u \\
   &\leq C_q e^{-r} \sum_{i=1}^{q-1}\sum_{j=0}^{i-1}\frac{q!}{j!(i-j)!(q-i)!}r^{j}\sum_{k=1}^{\infty}e^{-k}k^{i-j}+C_q e^{-r}\sum_{i=0}^{q-1}\frac{q!}{i!(q-i)!}r^{i}\sum_{k=0}^{\infty}e^{-k}.
\end{split}\label{intu2}
\end{align}
It follows from d'Alembert's ratio test that
\begin{align*}
    \max_{i=1,\cdots,q-1,\ j=0,\cdots,i-1}\left(\sum_{k=1}^{\infty}e^{-k}k^{i-j}\right)<\infty
\end{align*}
since one gets
\begin{align*}
    \biggl|\frac{e^{-{(k+1)}}(k+1)^{i-j}}{e^{-k}k^{i-j}}\biggr|=e^{-1}\biggl(1+\frac{1}{k}\biggr)^{i-j}\longrightarrow e^{-1}<1
\end{align*}
as $k\longrightarrow\infty$ for $i=1,\cdots,q-1$ and $j=0,\cdots,i-1$. Furthermore, we have
\begin{align*}
    \sum_{k=0}^{\infty}e^{-k}<\infty,
\end{align*}
so that it holds from (\ref{intu2}) that
\begin{align*}
    e^{-r}\sum_{k=0}^{\infty}e^{-k}\int_{\mathbb{U}_n(\theta_0)\cap\{r+k\leq |u|\leq r+k+1\}} \dd u &\leq C_q\sum_{i=1}^{q-1}\sum_{j=0}^{i-1}\frac{q!}{j!(i-j)!(q-i)!} e^{-r}r^{j}+C_q\sum_{i=0}^{q-1}\frac{q!}{i!(q-i)!}e^{-r}r^{i}\\
    &\leq C_q\sum_{i=0}^{q-1}\sum_{j=0}^{i}\frac{q!}{j!(i-j)!(q-i)!} e^{-r}r^{j} 
\end{align*}
for all $r>0$, which yields 
\begin{align}
\begin{split}
    &\quad\ {\bf{P}}\Biggl(\biggl\{\int_{\mathbb{U}_n(\theta_0)\cap\{|u|\geq r\}}\mathbb{Z}_n(u) \dd u >\varepsilon\biggr\}\cap\bigcap_{k=0}^{\infty}\biggl\{\sup_{u\in\mathbb{U}_n(\theta_0)\cap\{|u|\geq r+k\}}\mathbb{Z}_n(u)< e^{-(r+k)}\biggr\}\Biggr)\\
    &\leq {\bf{P}}\Biggl(C_q\sum_{i=0}^{q-1}\sum_{j=0}^{i}\frac{q!}{j!(i-j)!(q-i)!} e^{-r}r^{j} >\varepsilon\Biggr).
\end{split}\label{PZ2}
\end{align}
Hence, we see from (\ref{ZU})-(\ref{PZ1}) and (\ref{PZ2}) that
\begin{align}
    \begin{split}
    &\quad\ \sup_{n\in\mathbb{N}}{\bf{P}}\left(\int_{|u|\geq r}\mathbb{Z}_n(u) \dd u >\varepsilon\right)\\
    &\leq C_L\biggl\{\frac{1}{r^L}+\frac{1}{(L-1)r^{L-1}}\biggr\}+{\bf{P}}\Biggl(C_q\sum_{i=0}^{q-1}\sum_{j=0}^{i}\frac{q!}{j!(i-j)!(q-i)!} e^{-r}r^{j}>\varepsilon\Biggr)
    \end{split} \label{supP}
\end{align}
for all $r>0$. Since $e^{-r}r^{j}\longrightarrow 0$ as $r\longrightarrow \infty$ for any $j\in\mathbb{N}\cup\{0\}$, one gets
\begin{align*}
    C_L\biggl\{\frac{1}{r^L}+\frac{1}{(L-1)r^{L-1}}\biggr\}\longrightarrow 0,\quad
    C_q\sum_{i=0}^{q-1}\sum_{j=0}^{i}\frac{q!}{j!(i-j)!(q-i)!} e^{-r}r^{j}\longrightarrow 0
\end{align*}
as $r\longrightarrow\infty$. Therefore, there exists $r>0$ such that
\begin{align*}
    C_L\biggl\{\frac{1}{r^L}+\frac{1}{(L-1)r^{L-1}}\biggr\}<\varepsilon,\quad
    C_q\sum_{i=0}^{q-1}\sum_{j=0}^{i}\frac{q!}{j!(i-j)!(q-i)!} e^{-r}r^{j}<\varepsilon,
\end{align*}
so that it follows from (\ref{supP}) that
\begin{align*}
    \sup_{n\in\mathbb{N}}{\bf{P}}\left(\int_{|u|\geq r}\mathbb{Z}_n(u) \dd u >\varepsilon\right)<\varepsilon.
\end{align*}
\end{proof}
\begin{proof}[{\bf{Proof of Theorem \ref{Hexpansion}}}]
It is enough to check the regularity conditions of Theorem 3.7 in Eguchi and Masuda \cite{Eguchi(2018)} as $K=1$ and $A_{n}(\theta_0)=a_{1,n}(\theta_0)=n^{-\frac{1}{2}}$. Note that ${\bf{\Gamma}}_0$ is a positive definite matrix under {\bf{[C1]}} ({\rm{b}}); see Lemma 35 in Kusano and Uchida \cite{Kusano(2023)} for the proof.
Lemmas \ref{A1} and \ref{A3} satisfy Assumptions 3.1 and 3.3 in Eguchi and Masuda \cite{Eguchi(2018)} respectively. [{\bf{B}}] yields Assumption 3.2 in Eguchi and Masuda \cite{Eguchi(2018)}. Furthermore, it holds from Lemma \ref{thetahat} that $\hat{\theta}_n\stackrel{p}{\longrightarrow}\theta_0$, which completes the proof.
\end{proof}
\begin{lemma}\label{Huni}
Under {[\bf{A}]}, as $n\longrightarrow\infty$,
\begin{align*}
    n^{-1}\mathbb{H}_{m,n}(\theta_m)\stackrel{p}{\longrightarrow} \mathbb{H}_{m,0}(\theta_m)\quad \mbox{uniformly in }\theta_m
\end{align*}
for $m\in\{1,\cdots,M\}$.
\end{lemma}
\begin{proof}
Since
\begin{align*}
    n^{-1}\mathbb{H}_{m,n}(\theta_m)
    &=-2^{-1}{\bf{\Sigma}}_m(\theta_m)^{-1}\bigl[\mathbb{Q}_{\mathbb{XX}}\bigr]-2^{-1}\log\det{\bf{\Sigma}}_m(\theta_m)
\end{align*}
and
\begin{align*}
    \mathbb{H}_{m,0}(\theta_m)
    &=-2^{-1}{\bf{\Sigma}}_m
    (\theta_m)^{-1}\bigl[{\bf{\Sigma}}_0\bigr]-2^{-1}
    \log\det{\bf{\Sigma}}_m(\theta_m),
\end{align*}
it follows from Theorem $1$ in Kusano and Uchida \cite{Kusano(JJSD)} and the compactness of $\Theta$ that
\begin{align*}
    \sup_{\theta_m\in\Theta_m}\Bigl|n^{-1}\mathbb{H}_{m,n}(\theta_m)-\mathbb{H}_{m,0}(\theta_m)\Bigr|
    &\leq \sup_{\theta_m\in\Theta_m}\Bigl|{\bf{\Sigma}}_m(\theta_m)^{-1}\bigl[\mathbb{Q}_{\mathbb{XX}}-{\bf{\Sigma}}_0\bigr]\Bigr|\\
    &\leq \sup_{\theta_m\in\Theta_m}\Bigl|{\bf{\Sigma}}_m(\theta_m)^{-1}\Bigr|\bigl|\mathbb{Q}_{\mathbb{XX}}-{\bf{\Sigma}}_0\bigr|\stackrel{p}{\longrightarrow}0
\end{align*}
for $m\in\{1,\cdots,M\}$.
\end{proof}
\begin{lemma}\label{thetaasym}
Under {[\bf{A}]} and {[\bf{C1}]},
\begin{align}
    \log\det{\tilde{\bf{\Gamma}}_{m,n}(\hat{\theta}_{m,n})}\stackrel{p}{\longrightarrow}\log\det{\bf{\Gamma}}_{m,0} \label{logdet}
\end{align}
and 
\begin{align}
    \sqrt{n}(\hat{\theta}_{m,n}-\theta_{m,0})\stackrel{d}{\longrightarrow} N_{q_m}\Bigl(0,{\bf{\Gamma}}^{-1}_{m,0}\Bigr) \label{asymtheta}
\end{align}
for $m\in\mathfrak{M}$.
\end{lemma}
\begin{proof}
First, we will prove (\ref{logdet}). In a similar way to Lemma \ref{Huni}, it is shown that
\begin{align}
    n^{-1}\partial^2_{\theta_m}\mathbb{H}_{m,n}(\theta_m)\stackrel{p}{\longrightarrow}
    \partial^2_{\theta_m}\mathbb{H}_{m,0}(\theta_m)\quad \mbox{\rm{uniformly in}}\ \theta_m \label{H2uni}
\end{align}
for $m\in\mathfrak{M}$.
Since $\partial^2_{\theta_m}\mathbb{H}_{m,0}(\theta_m)$ is continuous in $\theta_m$, the continuous mapping theorem and Lemma \ref{thetahat} show
\begin{align}
    \partial^2_{\theta_m}\mathbb{H}_{m,0}(\hat{\theta}_{m,n})\stackrel{p}{\longrightarrow}\partial^2_{\theta_m}\mathbb{H}_{m,0}(\theta_{m,0})=-{\bf{\Gamma}}_{m,0}\label{H0conv}
\end{align}
for $m\in\mathfrak{M}$. Let 
\begin{align*}
    {\bf{\Gamma}}_{m,n}(\theta_m)=-n^{-1}\partial^2_{\theta_m}\mathbb{H}_{m,n}(\theta_m).
\end{align*}
It holds from (\ref{H2uni}) and (\ref{H0conv}) that
\begin{align*}
    \bigl|{\bf{\Gamma}}_{m,n}(\hat{\theta}_{m,n})-{\bf{\Gamma}}_{m,0}\bigr|&=\Bigl|-n^{-1}\partial^2_{\theta_m}\mathbb{H}_{m,n}(\hat{\theta}_{m,n})+\partial^2_{\theta_m}\mathbb{H}_{m,0}(\hat{\theta}_{m,n})
    -\partial^2_{\theta_m}\mathbb{H}_{m,0}(\hat{\theta}_{m,n})-{\bf{\Gamma}}_{m,0}\Bigr|\\
    &\leq \Bigl|n^{-1}\partial^2_{\theta_m}\mathbb{H}_{m,n}(\hat{\theta}_{m,n})-\partial^2_{\theta_m}\mathbb{H}_{m,0}(\hat{\theta}_{m,n})\Bigr|+\Bigl|
    \partial^2_{\theta_m}\mathbb{H}_{m,0}(\hat{\theta}_{m,n})+{\bf{\Gamma}}_{m,0}\Bigr|\\
    &\leq\sup_{\theta_m\in\Theta_m}
    \Bigl|n^{-1}\partial^2_{\theta_m}\mathbb{H}_{m,n}(\theta_m)-\partial^2_{\theta_m}\mathbb{H}_{m,0}(\theta_m)\Bigr|+\Bigl|
    \partial^2_{\theta_m}\mathbb{H}_{m,0}(\hat{\theta}_{m,n})+{\bf{\Gamma}}_{m,0}\Bigr|\\
    &\stackrel{p}{\longrightarrow} 0
\end{align*}
for $m\in\mathfrak{M}$, so that one gets
\begin{align*}
    {\bf{\Gamma}}_{m,n}(\hat{\theta}_{m,n})\stackrel{p}{\longrightarrow}{\bf{\Gamma}}_{m,0}.
\end{align*}
Note that ${\bf{P}}(J_{m,n})\longrightarrow 1$ as $n\longrightarrow\infty$ for $m\in\mathfrak{M}$ since ${\bf{\Gamma}}_{m,0}$ is a positive definite matrix. 
For all $\varepsilon>0$, we have
\begin{align*}
    {\bf{P}}\Bigl(\bigl|{\tilde{\bf{\Gamma}}}_{m,n}(\hat{\theta}_{m,n})-{\bf{\Gamma}}_{m,0}\bigr|>\varepsilon\Bigr)&\leq {\bf{P}}\biggl(\Bigl\{\bigl|{\tilde{\bf{\Gamma}}}_{m,n}(\hat{\theta}_{m,n})-{\bf{\Gamma}}_{m,0}\bigr|>\varepsilon\Bigr\}\cap J_{m,n}\biggr)\\
    &\qquad+{\bf{P}}\biggl(\Bigl\{\bigl|{\tilde{\bf{\Gamma}}}_{m,n}(\hat{\theta}_{m,n})-{\bf{\Gamma}}_{m,0}\bigr|>\varepsilon\Bigr\}\cap J^c_{m,n}\biggr)\\
    &\leq{\bf{P}}\biggl(\Bigl\{\bigl|{\bf{\Gamma}}_{m,n}(\hat{\theta}_{m,n})-{\bf{\Gamma}}_{m,0}\bigr|>\varepsilon\Bigr\}\cap J_{m,n}\biggr)+{\bf{P}}\Bigl(J_{m,n}^c\Bigr)\\
    &\leq {\bf{P}}\Bigl(\bigl|{\bf{\Gamma}}_{m,n}(\hat{\theta}_{m,n})-{\bf{\Gamma}}_{m,0}\bigr|>\varepsilon\Bigr)+{\bf{P}}\Bigl(J_{m,n}^c\Bigr)\longrightarrow 0
\end{align*}
as $n\longrightarrow\infty$ for $m\in\mathfrak{M}$, which yields 
\begin{align}
    {\tilde{\bf{\Gamma}}}_{m,n}(\hat{\theta}_{m,n})\stackrel{p}{\longrightarrow}{\bf{\Gamma}}_{m,0}. \label{gammaP}
\end{align}
Thus, it follows from the continuous mapping theorem and (\ref{gammaP}) that 
\begin{align*}
    \log\det{\tilde{\bf{\Gamma}}_{m,n}(\hat{\theta}_{m,n})}\stackrel{p}{\longrightarrow}\log\det{\bf{\Gamma}}_{m,0}
\end{align*}
for $m\in\mathfrak{M}$. For the proof of (\ref{asymtheta}), see Lemma $2$ in Kusano and Uchida \cite{Kusano(2024c)}.
\end{proof}
\begin{proof}[{\bf{Proof of Theorem \ref{BICcons1}}}]
In an analogous manner to the proof of Theorem \ref{Hexpansion}, Assumptions 3.1-3.3 in Eguchi and Masuda \cite{Eguchi(2018)} are satisfied. The proof of Theorem 5.1 (i) in Eguchi and Masuda \cite{Eguchi(2018)} is valid when $-n^{-1}\mathbb{H}_{m,n}(\hat{\theta}_{m,n})$ is a positive definite matrix. On the other hand, Eguchi and Masuda \cite{Eguchi(2018)} implies that ${\rm{QBIC}}_{1,n}$ and ${\rm{QBIC}}_{2,n}$ have the model selection consistency if $\tilde{\bf{\Gamma}}_{m,n}(\hat{\theta}_{m,n})\stackrel{p}{\longrightarrow}{\bf{\Gamma}}_{m,0}$ for $m\in\mathfrak{M}$. Indeed, from Lemma \ref{thetaasym}, we can obtain the result in the same way as the proof of Theorem 5.1 (i) in Eguchi and Masuda \cite{Eguchi(2018)}. See also Appendix \ref{Appendixproof}.
\end{proof}
\begin{proof}[{\bf{Proof of Theorem \ref{BICcons2}}}]
Set
\begin{align*}
    f({\bf{\Sigma}})=-2^{-1}{\bf{\Sigma}}^{-1}\bigl[{\bf{\Sigma}}_0\bigr]-2^{-1}\log\det{\bf{\Sigma}}
\end{align*}
for ${\bf{\Sigma}}\in\mathcal{M}_p^{+}$. Note that $ f({\bf{\Sigma}})$ has the unique maximum point at ${\bf{\Sigma}}={\bf{\Sigma}}_0$. Since we have
\begin{align*}
    {\bf{\Sigma}}_0={\bf{\Sigma}}_{m^*}(\theta_{m^*,0})\neq {\bf{\Sigma}}_m(\theta_{m})
\end{align*}
for $m\in\mathfrak{M}^c$, it is shown that
\begin{align}
    \mathbb{H}_{m,0}(\theta_m)=f\bigl({\bf{\Sigma}}_m
    (\theta_m)\bigr)<f({\bf{\Sigma}}_0)=f\bigl({\bf{\Sigma}}_{m^*}
    (\theta_{m^*,0})\bigr)=\mathbb{H}_{m^*,0}(\theta_{m^*,0}), \label{Hine}
\end{align}
which yields $\mathbb{H}_{m,0}(\theta_m)\neq\mathbb{H}_{m^*,0}(\theta_{m^*,0})$ for $m\in\mathfrak{M}^c$. For $m\in\{1,\cdots,M\}$, it holds from Lemma \ref{Huni} that $\mathbb{H}_{m,n}(\theta_m)\stackrel{p}{\longrightarrow}\mathbb{H}_{m,0}(\theta_m)\ \mbox{uniformly in }\theta_m$. Furthermore, {[\bf{C2}]} {\rm{(a)}} yields the uniqueness of $\bar{\theta}_m$. Therefore, in a similar way to the proof of Theorem 5.1 (ii) in Eguchi and Masuda \cite{Eguchi(2018)}, from Lemma \ref{thetaasym}, we can prove the result. See also Appendix \ref{Appendixproof}.
\end{proof}

\section{Appendix}
\subsection{The proofs of Theorem \ref{BICcons1} and Theorem \ref{BICcons2}}\label{Appendixproof}
\begin{proof}[\bf{Proof of Theorem \ref{BICcons1}}]
Let
\begin{align*}
    A_{m,n}=\Bigl\{\hat{\theta}_{m,n}\in{\rm{Int}}(\Theta_m)\Bigr\}.
\end{align*}
Since Model $m^*$ is nested in Model $m$, there exist a matrix $F\in\mathbb{R}^{q_m\times q_{m^*}}$ with $F^{\top}F=\mathbb{I}_{q_m^*}$ and a constant $c\in\mathbb{R}^{q_m}$ such that
\begin{align*}
    \mathbb{H}_{m^*,n}(\theta_{m^*})=\mathbb{H}_{m,n}(f_{m^*}(\theta_{m^*}))
\end{align*}
for all $\theta_{m^*}\in\Theta_{m^*}$, where $f_{m^*}(\theta_{m^*})=F\theta_{m^*}+c$.
Using the Taylor expansion, we have
\begin{align*}
    \mathbb{H}_{m^*,n}(\hat{\theta}_{m^*,n})
    &=\mathbb{H}_{m,n}(f_{m^*}(\hat{\theta}_{m^*,n}))\\
    &=\mathbb{H}_{m,n}(\hat{\theta}_{m,n})+\sqrt{n}\bigl(f_{m^*}(\hat{\theta}_{m^*,n})-\hat{\theta}_{m,n}\bigr)^{\top}{\rm{R}}_{n}\sqrt{n}\bigl(f_{m^*}(\hat{\theta}_{m^*,n})-\hat{\theta}_{m,n}\bigr)
\end{align*}
on $A_{m,n}$ for $m\in\mathfrak{M}\slash\{m^*\}$, where
\begin{align*}
    {\rm{R}}_{n}=n^{-1}\int_0^1 (1-\lambda)\partial_{\theta_m}^2\mathbb{H}_{m,n}(\tilde{\theta}_{m,m^*,\lambda,n})\dd\lambda.
\end{align*}
and
\begin{align*}
    \tilde{\theta}_{m,m^*,\lambda,n}=
    \hat{\theta}_{m,n}+\lambda\bigl(f_{m^*}(\hat{\theta}_{m^*,n})-\hat{\theta}_{m,n}\bigr).
\end{align*}
Note that
\begin{align*}
    f_{m^*}(\hat{\theta}_{m^*,n})-\theta_{m,0}&=f_{m^*}(\hat{\theta}_{m^*,n})-f_{m^*}(\theta_{m^*,0})\\
    &=F\hat{\theta}_{m^*,n}+c-(F\theta_{m^*,0}+c)=F(\hat{\theta}_{m^*,n}-\theta_{m^*,0})
\end{align*}
for $m\in\mathfrak{M}\slash\{m^*\}$. Lemma \ref{thetaasym} yields
\begin{align}
\begin{split}
    \sqrt{n}\bigl(f_{m^*}(\hat{\theta}_{m^*,n})-\hat{\theta}_{m,n}\bigr)&=\sqrt{n}\bigl(f_{m^*}(\hat{\theta}_{m^*,n})-\theta_{m,0}\bigr)-\sqrt{n}\bigl(\hat{\theta}_{m,n}-\theta_{m,0}\bigr)\\
    &=F\Bigl\{\sqrt{n}(\hat{\theta}_{m^*,n}-\theta_{m^*,0})\Bigr\}-\sqrt{n}\bigl(\hat{\theta}_{m,n}-\theta_{m,0}\bigr)=O_p(1)
\end{split}\label{fthetaOp1}
\end{align}
for $m\in\mathfrak{M}\slash\{m^*\}$. Furthermore, using the Taylor expansion, it is shown that
\begin{align*}
    \partial_{\theta^{(i)}_m}\partial_{\theta^{(j)}_m}
    \mathbb{H}_{m,n}(\tilde{\theta}_{m,m^*,\lambda,n})&=\partial_{\theta^{(i)}_m}\partial_{\theta^{(j)}_m}\mathbb{H}_{m,n}(\theta_{m,0})\\
    &\quad\qquad+\sum_{k=1}^q\biggl(\int_{0}^{1}
    \partial_{\theta_m^{(i)}}\partial_{\theta_m^{(j)}}
    \partial_{\theta_m^{(k)}}
    \mathbb{H}_{m,n}(\check{\theta}_{m,m^*,\lambda,s,n}
    )\dd s\biggr)\bigl(\tilde{\theta}^{(k)}_{m,m^*,\lambda,n}-\theta^{(k)}_{m,0}\bigr)\\
    &=\partial_{\theta^{(i)}_m}\partial_{\theta^{(j)}_m}\mathbb{H}_{m,n}(\theta_{m,0})\\
    &\qquad\quad+\sum_{k=1}^q\biggl(\int_{0}^{1}
    \partial_{\theta_m^{(i)}}\partial_{\theta_m^{(j)}}
    \partial_{\theta_m^{(k)}}
    \mathbb{H}_{m,n}(\check{\theta}_{m,m^*,\lambda,s,n}
    )\dd s\biggr)\bigl(\hat{\theta}^{(k)}_{m,n}-\theta^{(k)}_{m,0}\bigr)\\
    &\qquad\quad+\sum_{k=1}^q\biggl(\int_{0}^{1}
    \partial_{\theta_m^{(i)}}\partial_{\theta_m^{(j)}}
    \partial_{\theta_m^{(k)}}
    \mathbb{H}_{m,n}(\check{\theta}_{m,m^*,\lambda,s,n})\dd s\biggr)\lambda\bigl(f_{m^*}(\hat{\theta}_{m^*,n})^{(k)}-\hat{\theta}^{(k)}_{m,n}\bigr)
\end{align*}
for $i,j=1,\cdots,q$ and $m\in\mathfrak{M}\slash\{m^*\}$, where
\begin{align*}
    \check{\theta}_{m,m^*,\lambda,s,n}=\theta_{m,0}+
    s\bigl(\tilde{\theta}_{m,m^*,\lambda,n}-\theta_{m,0}\bigr).
\end{align*}
Consequently, we obtain
\begin{align*}
     \bigl({\rm{R}}_{n}\bigr)_{ij}
     &=n^{-1}\int_0^1 (1-\lambda)\partial_{\theta^{(i)}_m}\partial_{\theta^{(j)}_m}\mathbb{H}_{m,n}(\tilde{\theta}_{m,m^*,\lambda,n})\dd\lambda\\
    &=2^{-1}n^{-1}\partial_{\theta^{(i)}_m}\partial_{\theta^{(j)}_m}
    \mathbb{H}_{m,n}(\theta_{m,0})\\
    &\quad+\sum_{k=1}^q\int_{0}^{1}(1-\lambda)\int_{0}^{1}
    n^{-\frac{3}{2}}\partial_{\theta_m^{(i)}}\partial_{\theta_m^{(j)}}
    \partial_{\theta_m^{(k)}}
    \mathbb{H}_{m,n}(\check{\theta}_{m,m^*,\lambda,s,n}
    )\dd s \dd\lambda\sqrt{n}\bigl(\hat{\theta}^{(k)}_{m,n}-\theta^{(k)}_{m,0}\bigr)\\    
    &\quad+\sum_{k=1}^q\int_{0}^{1}(1-\lambda)\lambda\int_{0}^{1}n^{-\frac{3}{2}}
    \partial_{\theta_m^{(i)}}\partial_{\theta_m^{(j)}}
    \partial_{\theta_m^{(k)}}
    \mathbb{H}_{m,n}(\check{\theta}_{m,m^*,\lambda,s,n})\dd s\dd\lambda\sqrt{n}\bigl(f_{m^*}(\hat{\theta}_{m^*,n})^{(k)}-\hat{\theta}^{(k)}_{m,n}\bigr)
\end{align*}
for $i,j=1,\cdots,q$ and $m\in\mathfrak{M}\slash\{m^*\}$. Lemma \ref{A1} and Lemma \ref{thetaasym} show
\begin{align}
\begin{split}
    &\quad\ \Biggl|\sum_{k=1}^q\int_{0}^{1}(1-\lambda)\int_{0}^{1}
    n^{-\frac{3}{2}}\partial_{\theta_m^{(i)}}\partial_{\theta_m^{(j)}}
    \partial_{\theta_m^{(k)}}
    \mathbb{H}_{m,n}(\check{\theta}_{m,m^*,\lambda,s,n}
    )\dd s \dd\lambda\sqrt{n}\bigl(\hat{\theta}^{(k)}_{m,n}-\theta^{(k)}_{m,0}\bigr)\Biggr|\\
    &\leq \sum_{k=1}^q\biggl\{\int_{0}^{1}(1-\lambda)\int_{0}^{1}
    \Bigl|n^{-\frac{3}{2}}\partial_{\theta_m^{(i)}}\partial_{\theta_m^{(j)}}
    \partial_{\theta_m^{(k)}}
    \mathbb{H}_{m,n}(\check{\theta}_{m,m^*,\lambda,s,n})\Bigr|\dd s \dd\lambda\biggr\}\Bigl|\sqrt{n}\bigl(\hat{\theta}^{(k)}_{m,n}-\theta^{(k)}_{m,0}\bigr)\Bigr|\\
    &\leq\biggl(\int_0^1 (1-\lambda)d\lambda\biggr)\sum_{k=1}^q\sup_{\theta_m\in\Theta_m}\Bigl|n^{-\frac{3}{2}}\partial_{\theta_m^{(i)}}\partial_{\theta_m^{(j)}}
    \partial_{\theta_m^{(k)}}
    \mathbb{H}_{m,n}(\theta_m)\Bigr|\Bigl|\sqrt{n}\bigl(\hat{\theta}^{(k)}_{m,n}-\theta^{(k)}_{m,0}\bigr)\Bigr|\\
    &\stackrel{p}{\longrightarrow}0
\end{split}\label{intP1}
\end{align}
for $i,j=1,\cdots,q$ and $m\in\mathfrak{M}\slash\{m^*\}$. In a similar way, it holds from Lemma \ref{A1}, Lemma \ref{thetaasym} and (\ref{fthetaOp1}) that
\begin{align}
\begin{split}
    &\quad\ \Biggl|\sum_{k=1}^q\int_{0}^{1}(1-\lambda)\lambda\int_{0}^{1}
    n^{-\frac{3}{2}}\partial_{\theta_m^{(i)}}\partial_{\theta_m^{(j)}}
    \partial_{\theta_m^{(k)}}
    \mathbb{H}_n(\check{\theta}_{m,m^*,\lambda,s,n}
    )\dd s\dd\lambda\sqrt{n}\bigl(f_{m^*}(\hat{\theta}_{m^*,n})^{(k)}-\hat{\theta}^{(k)}_{m,n}\bigr)\Biggr|\\
    &\leq\biggl(\int_0^1 (1-\lambda)\lambda d\lambda\biggr)\sum_{k=1}^q\sup_{\theta_m\in\Theta_m}\Bigl|n^{-\frac{3}{2}}\partial_{\theta_m^{(i)}}\partial_{\theta_m^{(j)}}
    \partial_{\theta_m^{(k)}}
    \mathbb{H}_n(\theta_m)\Bigr|
    \Bigl|\sqrt{n}\bigl(f_{m^*}(\hat{\theta}_{m^*,n})^{(k)}-\hat{\theta}^{(k)}_{m,n}\bigr)\Bigr|\\
    &\stackrel{p}{\longrightarrow}0
\end{split}\label{intP2}
\end{align}
for $i,j=1,\cdots,q$ and $m\in\mathfrak{M}\slash\{m^*\}$. Thus, it follows from Lemma \ref{A1}, (\ref{intP1}) and (\ref{intP2}) that 
\begin{align}
    {\rm{R}}_{n}\stackrel{p}{\longrightarrow}-2^{-1}{\bf{\Gamma}}_{m,0}. \label{Rconv}
\end{align}
Furthermore, we see from Lemma \ref{thetaasym}, (\ref{fthetaOp1}) and (\ref{Rconv}) that
\begin{align}
\begin{split}
    &\quad\ {\rm{QBIC}}^{(m^*)}_{1,n}-{\rm{QBIC}}^{(m)}_{1,n}\\
    &=-2\mathbb{H}_{m^*,n}(\hat{\theta}_{m^*,n})+2\mathbb{H}_{m,n}(\hat{\theta}_{m,n})+\log\det{n\tilde{\bf{\Gamma}}_{m^*,n}(\hat{\theta}_{m^*,n})}
    -\log\det{n\tilde{\bf{\Gamma}}_{m,n}(\hat{\theta}_{m,n})}\\
    &=-2\sqrt{n}\bigl(f_{m^*}(\hat{\theta}_{m^*,n})-\hat{\theta}_{m,n}\bigr)^{\top}{\rm{R}}_{n}\sqrt{n}\bigl(f_{m^*}(\hat{\theta}_{m^*,n})-\hat{\theta}_{m,n}\bigr)\\
    &\qquad\qquad\qquad\qquad
    +\log\det{\tilde{\bf{\Gamma}}_{m^*,n}(\hat{\theta}_{m^*,n})}-\log\det{\tilde{\bf{\Gamma}}_{m,n}(\hat{\theta}_{m,n})}-(q_m-q_{m^*})\log n
\end{split}\label{BICine}
\end{align}
on $A_{m,n}$, and
\begin{align}
\begin{split}
    &-2\sqrt{n}\bigl(f_{m^*}(\hat{\theta}_{m^*,n})-\hat{\theta}_{m,n}\bigr)^{\top}{\rm{R}}_{n}\sqrt{n}\bigl(f_{m^*}(\hat{\theta}_{m^*,n})-\hat{\theta}_{m,n}\bigr)\\
    &\qquad\qquad\qquad\qquad\qquad
    +\log\det{\tilde{\bf{\Gamma}}_{m^*,n}(\hat{\theta}_{m^*,n})}-\log\det{\tilde{\bf{\Gamma}}_{m,n}(\hat{\theta}_{m,n})}=O_p(1).
\end{split}\label{BICOp1}
\end{align}
Therefore, it follows from (\ref{BICine}), (\ref{BICOp1}) and Slutsuky's Theorem that 
\begin{align*}
    &\quad\ {\bf{P}}\Bigl({\rm{QBIC}}^{(m^*)}_{1,n}<{\rm{QBIC}}^{(m)}_{1,n}\Bigr)\\
    &={\bf{P}}\biggl(\Bigl\{{\rm{QBIC}}^{(m^*)}_{1,n}<{\rm{QBIC}}^{(m)}_{1,n}\Bigr\}\cap A_{m,n}\biggr)+{\bf{P}}\biggl(\Bigl\{{\rm{QBIC}}^{(m^*)}_{1,n}<{\rm{QBIC}}^{(m)}_{1,n}\Bigr\}\cap A_{m,n}^c\biggr)\\
    &={\bf{P}}\left(\Biggl\{-2\sqrt{n}\bigl(f_{m^*}(\hat{\theta}_{m^*,n})-\hat{\theta}_{m,n}\bigr)^{\top}{\rm{R}}_{n}\sqrt{n}\bigl(f_{m^*}(\hat{\theta}_{m^*,n})-\hat{\theta}_{m,n}\bigr)\right.\\
    &\left.\qquad\qquad+\log\det{\tilde{\bf{\Gamma}}_{m^*,n}(\hat{\theta}_{m^*,n})}-\log\det{\tilde{\bf{\Gamma}}_{m,n}(\hat{\theta}_{m,n})}
    <(q_m-q_{m^*})\log n\Biggr\}\cap A_{m,n}\right)+{\bf{P}}\Bigl(A_{m,n}^c\Bigr)\\
    &={\bf{P}}\Biggl((q_m-q_{m^*})^{-1}(\log n)^{-1}\biggl\{-2\sqrt{n}\bigl(f_{m^*}(\hat{\theta}_{m^*,n})-\hat{\theta}_{m,n}\bigr)^{\top}{\rm{R}}_{n}\sqrt{n}\bigl(f_{m^*}(\hat{\theta}_{m^*,n})-\hat{\theta}_{m,n}\bigr)\\
    &\qquad\qquad\qquad\qquad\qquad\quad+\log\det{\tilde{\bf{\Gamma}}_{m^*,n}(\hat{\theta}_{m^*,n})}-\log\det{\tilde{\bf{\Gamma}}_{m,n}(\hat{\theta}_{m,n})}<1\biggr\}\Biggr)+{\bf{P}}\Bigl(A_{m,n}^c\Bigr)\longrightarrow 1
\end{align*}
as $n\longrightarrow\infty$ for $m\in\mathfrak{M}\slash\{m^*\}$. In a similar way, we can obtain
\begin{align*}
    {\bf{P}}\Bigl({\rm{QBIC}}^{(m^*)}_{2,n}<{\rm{QBIC}}^{(m)}_{2,n}\Bigr)\longrightarrow 1
\end{align*}
as $n\longrightarrow\infty$ for $m\in\mathfrak{M}\slash\{m^*\}$.
\end{proof}
\begin{lemma}\label{optcons}
Under {[\bf{A}]} and {[\bf{C2}]} {\rm{(a)}}, as $n\longrightarrow\infty$,
\begin{align*}
    \hat{\theta}_{m,n}\stackrel{p}{\longrightarrow}\bar{\theta}_{m}
\end{align*}
for $m\in\{1,\cdots,M\}$.
\end{lemma}
\begin{proof}
    The result can be shown in a similar way to Lemma 36 in Kusano and Uchida \cite{Kusano(2023)}.
\end{proof}
\begin{proof}[\bf{Proof of Theorem \ref{BICcons2}}]
Note that $\mathbb{H}_{m,0}(\theta_{m})$ is continuous in $\theta_m\in\Theta_m$. It holds from Lemma \ref{optcons} and the continuous mapping theorem that 
\begin{align}
    \mathbb{H}_{m,0}(\hat{\theta}_{m,n})\stackrel{p}{\longrightarrow}\mathbb{H}_{m,0}(\bar{\theta}_{m}) \label{Hp}
\end{align}
for $m\in\{1,\cdots,M\}$. Thus, it follows from Lemma \ref{Huni} and (\ref{Hp}) that 
\begin{align*}
    \Bigl|n^{-1}\mathbb{H}_{m,n}(\hat{\theta}_{m,n})-\mathbb{H}_{m,0}(\bar{\theta}_{m})\Bigr|
    &\leq \Bigl|n^{-1}\mathbb{H}_{m,n}(\hat{\theta}_{m,n})-\mathbb{H}_{m,0}(\hat{\theta}_{m,n})\Bigr|+\Bigl|\mathbb{H}_{m,0}(\hat{\theta}_{m,n})-\mathbb{H}_{m,0}(\bar{\theta}_{m})\Bigr|\\
    &\leq\sup_{\theta_{m}\in\Theta_{m}}\Bigl|n^{-1}\mathbb{H}_{m,n}(\theta_{m})-\mathbb{H}_{m,0}(\theta_{m})\Bigr|+\Bigl|\mathbb{H}_{m,0}(\hat{\theta}_{m,n})-\mathbb{H}_{m,0}(\bar{\theta}_{m})\Bigr|\\
    &\stackrel{p}{\longrightarrow}0
\end{align*}
for $m\in\{1,\cdots,M\}$, which yields
\begin{align}
    n^{-1}\mathbb{H}_{m,n}(\hat{\theta}_{m,n})\stackrel{p}{\longrightarrow}\mathbb{H}_{m,0}(\bar{\theta}_{m}).
    \label{Hmprob}
\end{align}
In an analogous manner to Lemma \ref{thetaasym}, it is shown that
\begin{align}
    \log\det{\tilde{\bf{\Gamma}}_{m,n}(\hat{\theta}_{m,n})}\stackrel{p}{\longrightarrow}\log\det{\Bigl\{-\partial^2_{\theta_m}\mathbb{H}_{m,0}(\bar{\theta}_m})\Bigr\} \label{logdet2}
\end{align}
for $m\in\{1,\cdots,M\}$. For $m\in\mathfrak{M}^c$, we see from (\ref{Hmprob}) and (\ref{logdet2}) that
\begin{align}
\begin{split}
    &\quad\ n^{-1}{\rm{QBIC}}^{(m^*)}_{1,n}-n^{-1}{\rm{QBIC}}^{(m)}_{1,n}\\
    &=-2n^{-1}\mathbb{H}_{m^*,n}(\hat{\theta}_{m^*,n})+2n^{-1}\mathbb{H}_{m,n}(\hat{\theta}_{m,n})\\
    &\qquad\qquad\qquad\qquad\qquad+n^{-1}\log\det{n\tilde{\bf{\Gamma}}_{m^*,n}(\hat{\theta}_{m^*,n})}
    -n^{-1}\log\det{n\tilde{\bf{\Gamma}}_{m,n}(\hat{\theta}_{m,n})}\\
    &=-2n^{-1}\mathbb{H}_{m^*,n}(\hat{\theta}_{m^*,n})+2n^{-1}\mathbb{H}_{m,n}(\hat{\theta}_{m,n})+n^{-1}\log\det{\tilde{\bf{\Gamma}}_{m^*,n}(\hat{\theta}_{m^*,n})}\\
    &\qquad\qquad\qquad\qquad\qquad\qquad\qquad-n^{-1}\log\det{\tilde{\bf{\Gamma}}_{m,n}(\hat{\theta}_{m,n})}-(q_m-q_{m^*})n^{-1}\log n\\
    &\stackrel{p}{\longrightarrow}-2\mathbb{H}_{m^*,0}(\theta_{m^*,0})
    +2\mathbb{H}_{m,0}(\bar{\theta}_{m}).
\end{split}\label{BICprob}
\end{align}
Furthermore, (\ref{Hine}) shows
\begin{align}
    2\mathbb{H}_{m^*,0}(\theta_{m^*,0})-2\mathbb{H}_{m,0}(\bar{\theta}_{m})>0
    \label{Hpos}
\end{align}
for $m\in\mathfrak{M}^c$. Therefore, it holds from (\ref{BICprob}) and (\ref{Hpos}) that 
\begin{align*}
    &\quad\ {\bf{P}}\Bigl({\rm{QBIC}}^{(m^*)}_{1,n}<{\rm{QBIC}}^{(m)}_{1,n}\Bigr)\\
    &={\bf{P}}\Bigl(n^{-1}{\rm{QBIC}}^{(m^*)}_{1,n}-n^{-1}{\rm{QBIC}}^{(m)}_{1,n}+2\mathbb{H}_{m^*,0}(\theta_{m^*,0})-2\mathbb{H}_{m,0}(\bar{\theta}_{m})<2\mathbb{H}_{m^*,0}(\theta_{m^*,0})-2\mathbb{H}_{m,0}(\bar{\theta}_{m})\Bigr)\\
    &\longrightarrow 1
\end{align*}
as $n\longrightarrow\infty$ for $m\in\mathfrak{M}^c$. In a similar way, we can prove 
\begin{align*}
    {\bf{P}}\Bigl({\rm{QBIC}}^{(m^*)}_{2,n}<{\rm{QBIC}}^{(m)}_{2,n}\Bigr)\longrightarrow 1
\end{align*}
as $n\longrightarrow\infty$ for $m\in\mathfrak{M}^c$.
\end{proof}
\subsection{The check of identifiability}\label{idensec}
\setlength{\abovedisplayskip}{8pt}
\setlength{\belowdisplayskip}{8pt}
In a similar way to Appendix 6.2 in Kusano and Uchida \cite{Kusano(2024c)}, to check [{\bf{C1}}] ({\rm{a}}), it is sufficient to show
\begin{align*}
    {\bf{\Sigma}}_1(\theta_{1})={\bf{\Sigma}}_1(\theta_{1,0})\Longrightarrow \theta_1=\theta_{1,0} 
\end{align*}
and ${\bf{\Gamma}}_{1,0}>0$. Since $\rank\Delta_{1,0}=22$, we see from Lemma 35 in Kusano and Uchida \cite{Kusano(2023)} that ${\bf{\Gamma}}_{1,0}>0$. Assume
\begin{align}
    {\bf{\Sigma}}_1(\theta_{1})={\bf{\Sigma}}_1(\theta_{1,0}).
    \label{iden}
\end{align}
Note that $\theta^{(1)}_1$, $\theta^{(2)}_1$ and $\theta^{(10)}_1$ are not zero. The (1,2)-th, (1,3)-th, and (2,3)-th elements of  (\ref{iden}) show
\begin{align}
\begin{split}
    \theta^{(10)}_1
    &=\theta^{(1)}_1\theta^{(10)}_1\times \frac{1}{\theta^{(1)}_1\theta^{(2)}_1\theta^{(10)}_1}\times \theta^{(2)}_1\theta^{(10)}_1\\
    &={\bf{\Sigma}}_1(\theta_1)_{12}\times 
    \frac{1}{{\bf{\Sigma}}_1(\theta_1)_{23}}\times {\bf{\Sigma}}_1(\theta_1)_{13}\\
    &={\bf{\Sigma}}_1(\theta_{1,0})_{12}\times 
    \frac{1}{{\bf{\Sigma}}_1(\theta_{1,0})_{23}}\times {\bf{\Sigma}}_1(\theta_{1,0})_{13}\\
    &=\theta^{(1)}_{1,0}\theta^{(10)}_{1,0}\times \frac{1}{\theta^{(1)}_{1,0}\theta^{(2)}_{1,0}\theta^{(10)}_{1,0}}\times \theta^{(2)}_{1,0}\theta^{(10)}_{1,0}=\theta^{(10)}_{1,0}.
\end{split}\label{10}
\end{align}
Thus, it holds from (\ref{10}) and the (1,2)-th, (1,3)-th, and (1,4)-th elements of (\ref{iden}) that
\begin{align*}
    \bigl(\theta^{(1)}_1-\theta_{1,0}^{(1)}\bigr)\theta^{(10)}_{1,0}=0,\quad \bigl(\theta^{(2)}_1-\theta_{1,0}^{(2)}\bigr)\theta^{(10)}_{1,0}=0,\quad \bigl(\theta^{(3)}_1-\theta_{1,0}^{(3)}\bigr)\theta^{(10)}_{1,0}=0
\end{align*}
and $\theta^{(10)}_{1,0}$ is not zero, which yields
\begin{align}
    \theta^{(1)}_1=\theta_{1,0}^{(1)},\quad 
    \theta^{(2)}_1=\theta_{1,0}^{(2)},\quad
    \theta^{(3)}_1=\theta_{1,0}^{(3)}.\label{1-3}
\end{align}
It follows from the (1,1)-th, (2,2)-th, (3,3)-th, and (4,4)-th elements of (\ref{iden}) that
\begin{align*}
    \theta^{(10)}_1+\theta^{(11)}_1&=
    \theta^{(10)}_{1,0}+\theta^{(11)}_{1,0},\\
    \theta^{(1)2}_1\theta^{(10)}_1+\theta^{(12)}_1&=
    \theta^{(1)2}_{1,0}\theta^{(10)}_{1,0}
    +\theta^{(12)}_{1,0},\\
    \theta^{(2)2}_1\theta^{(10)}_1+\theta^{(13)}_1&=
    \theta^{(2)2}_{1,0}\theta^{(10)}_{1,0}
    +\theta^{(13)}_{1,0},\\
    \theta^{(3)2}_1\theta^{(10)}_1+\theta^{(14)}_1&=
    \theta^{(3)2}_{1,0}\theta^{(10)}_{1,0}
    +\theta^{(14)}_{1,0},
\end{align*}
so that (\ref{10}) and (\ref{1-3}) show
\begin{align}
    \theta^{(11)}_1=\theta_{1,0}^{(11)},\quad
    \theta^{(12)}_1=\theta_{1,0}^{(12)},\quad 
    \theta^{(13)}_1=\theta_{1,0}^{(13)},\quad
    \theta^{(14)}_1=\theta_{1,0}^{(14)}.\label{11-14}
\end{align}
Moreover, we see from (\ref{10}) and the (1,5)-th, and (1,8)-th elements of (\ref{iden}) that 
\begin{align*}
    (\theta_1^{(8)}-\theta_{1,0}^{(8)})\theta_{1,0}^{(10)}=0,\quad
    (\theta_1^{(9)}-\theta_{1,0}^{(9)})\theta_{1,0}^{(10)}=0
\end{align*}
and $\theta^{(10)}_{1,0}$ is not zero, which implies
\begin{align}
    \theta_1^{(8)}=\theta_{1,0}^{(8)},\quad
    \theta_1^{(9)}=\theta_{1,0}^{(9)}. \label{8-9}
\end{align}
Since it follows from (\ref{10}), (\ref{8-9}), and the (1.6)-th, (1.7)-th, (1.9)-th, and (1.10)-th elements of (\ref{iden}) that
\begin{align*}
    (\theta^{(4)}_{1}-\theta^{(4)}_{1,0})\theta_{1,0}^{(8)}\theta_{1,0}^{(10)}&=0,\quad
    (\theta^{(5)}_{1}-\theta^{(5)}_{1,0})\theta_{1,0}^{(8)}\theta_{1,0}^{(10)}=0,\\
    (\theta^{(6)}_{1}-\theta^{(6)}_{1,0})\theta_{1,0}^{(9)}\theta_{1,0}^{(10)}&=0,\quad
    (\theta^{(7)}_{1}-\theta^{(7)}_{1,0})\theta_{1,0}^{(9)}\theta_{1,0}^{(10)}=0
\end{align*}
and $\theta_{1,0}^{(8)}$, $\theta_{1,0}^{(9)}$ and $\theta_{1,0}^{(10)}$ are not zero, one gets
\begin{align}
    \theta_1^{(4)}=\theta_{1,0}^{(4)},\quad
    \theta_1^{(5)}=\theta_{1,0}^{(5)},\quad
    \theta_1^{(6)}=\theta_{1,0}^{(6)},\quad
    \theta_1^{(7)}=\theta_{1,0}^{(7)}.\label{4-7}
\end{align}
Note that the (5.6)-th, and (8.9)-th elements of (\ref{iden}) are
\begin{align*}
    \theta_1^{(4)}\theta_1^{(8)2}\theta_1^{(10)}
    +\theta_1^{(4)}\theta_1^{(21)}
    &=\theta_{1,0}^{(4)}
    \theta_{1,0}^{(8)2}\theta_{1,0}^{(10)}
    +\theta_{1,0}^{(4)}\theta_{1,0}^{(21)},\\
    \theta_1^{(6)}\theta_1^{(9)2}\theta_1^{(10)}
    +\theta_1^{(6)}\theta_1^{(22)}
    &=\theta_{1,0}^{(6)}
    \theta_{1,0}^{(9)2}\theta_{1,0}^{(10)}
    +\theta_{1,0}^{(6)}\theta_{1,0}^{(22)}
\end{align*}
and $\theta_{1,0}^{(4)}$ and $\theta_{1,0}^{(6)}$ are not zero. (\ref{10}), (\ref{8-9}) and (\ref{4-7}) implies 
\begin{align*}
    \theta_{1,0}^{(4)}(\theta_1^{(21)}-\theta_{1,0}^{(21)})=0,\quad \theta_{1,0}^{(6)}(\theta_1^{(22)}-\theta_{1,0}^{(22)})=0,
\end{align*}
which yields 
\begin{align}
    \theta_{1}^{(21)}=\theta_{1,0}^{(21)},\quad \theta_{1}^{(22)}=\theta_{1,0}^{(22)}. \label{21-22}
\end{align}
As the (5.5)-th, (6.6)-th, (7.7)-th, (8.8)-th, (9.9)-th and (10.10)-th elements of (\ref{iden}) are
\begin{align*}
    \theta_1^{(8)2}\theta_1^{(10)}+\theta_1^{(21)}+
    \theta_1^{(15)}&=
    \theta_{1,0}^{(8)2}
    \theta_{1,0}^{(10)}+\theta_{1,0}^{(21)}+
    \theta_{1,0}^{(15)},\\
    \theta_1^{(4)2}
    \theta_1^{(8)2}\theta_1^{(10)}+\theta_1^{(4)2}
    \theta_1^{(21)}+
    \theta_1^{(16)}&=
    \theta_{1,0}^{(4)2}\theta_{1,0}^{(8)2}
    \theta_{1,0}^{(10)}+\theta_{1,0}^{(4)2}
    \theta_{1,0}^{(21)}+
    \theta_{1,0}^{(16)},\\
    \theta_1^{(5)2}
    \theta_1^{(8)2}\theta_1^{(10)}+\theta_1^{(5)2}
    \theta_1^{(21)}+
    \theta_1^{(17)}&=
    \theta_{1,0}^{(5)2}\theta_{1,0}^{(8)2}
    \theta_{1,0}^{(10)}+\theta_{1,0}^{(5)2}
    \theta_{1,0}^{(21)}+
    \theta_{1,0}^{(17)},\\
    \theta_1^{(9)2}\theta_1^{(10)}+\theta_1^{(22)}+
    \theta_1^{(18)}&=
    \theta_{1,0}^{(9)2}
    \theta_{1,0}^{(10)}+\theta_{1,0}^{(22)}+
    \theta_{1,0}^{(18)},\\
    \theta_1^{(6)2}
    \theta_1^{(9)2}\theta_1^{(10)}+\theta_1^{(6)2}
    \theta_1^{(22)}+
    \theta_1^{(19)}&=
    \theta_{1,0}^{(6)2}\theta_{1,0}^{(9)2}
    \theta_{1,0}^{(10)}+\theta_{1,0}^{(6)2}
    \theta_{1,0}^{(22)}+
    \theta_{1,0}^{(19)},\\
    \theta_1^{(7)2}
    \theta_1^{(9)2}\theta_1^{(10)}+\theta_1^{(7)2}
    \theta_1^{(22)}+
    \theta_1^{(20)}&=
    \theta_{1,0}^{(7)2}\theta_{1,0}^{(9)2}
    \theta_{1,0}^{(10)}+\theta_{1,0}^{(7)2}
    \theta_{1,0}^{(22)}+
    \theta_{1,0}^{(20)},
\end{align*}
we see from (\ref{10}), (\ref{8-9}), (\ref{4-7}) and (\ref{21-22}) that
\begin{align}
    \theta_1^{(15)}=\theta_{1,0}^{(15)},\quad
    \theta_1^{(16)}=\theta_{1,0}^{(16)},\quad
    \theta_1^{(17)}=\theta_{1,0}^{(17)},\quad
    \theta_1^{(18)}=\theta_{1,0}^{(18)},\quad
    \theta_1^{(19)}=\theta_{1,0}^{(19)},\quad
    \theta_1^{(20)}=\theta_{1,0}^{(20)}. \label{15-20}
\end{align}
Therefore, it follows from (\ref{10})-(\ref{15-20}) that $\theta_{1}=\theta_{1,0}$, which completes the proof.
\end{document}